\documentclass[11pt,a4paper,reqno]{amsart}
\usepackage[applemac,utf8]{inputenc}
\usepackage[T1]{fontenc}
\usepackage{ulem}
\usepackage{amsmath,esint,color}
\usepackage{amsthm}
\usepackage{amsfonts}
\usepackage{amssymb}
\usepackage{bigints}
\usepackage{fullpage}
\usepackage[colorlinks=true, pdfstartview=FitV, linkcolor=blue, citecolor=blue, urlcolor=blue]{hyperref}

\date{}
\definecolor{sah}{rgb}{0.66,0.33, 0.04}
\definecolor{adel4}{cmyk}{1,0,0,0}
\definecolor{adel3}{rgb}{0.66,0.33, 0.04}
\definecolor{adel1}{cmyk}{0,0.20,1,0}
\definecolor{adel2}{cmyk}{0,0.40,1,0.30}
\definecolor{adel0}{rgb}{0.99,0.60, 0.30}
\definecolor{trut}{rgb}{0.99,0.80, 0.00}
\definecolor{trus}{rgb}{0.00, 0.50, 0.00}
 \definecolor{trust}{rgb}{0.99, 0.99, 0.80}
\definecolor{MaCouleur}{rgb}{0,0.9,0.3}

\newcommand{\NN}{\mathbb{N}}  
\newcommand{\RR}{\mathbb{R}}  

\newcommand{\pxl}{\partial_{X_\lambda}}
\newcommand{\pxzl}{\partial_{X_{0,\lambda}}}
\newcommand{\pxtl}{\partial_{X_{t,\lambda}}}
\newcommand{\pxt}{\partial_{X_{t}}}
\newcommand{\pxz}{\partial_{X_{0}}}
\theoremstyle{plain}
\newtheorem{definition}{Definition}
\newtheorem{theorem}{Theorem}
\newtheorem{proposition}{Proposition}
\newtheorem{lemma}{Lemma}
\newtheorem{remark}{Remark}
\newtheorem{coro}{Corollary}

\def\virgp{\raise 2pt\hbox{,}}

\def\Xint#1{\mathchoice
   {\XXint\displaystyle\textstyle{#1}}%
   {\XXint\textstyle\scriptstyle{#1}}%
   {\XXint\scriptstyle\scriptscriptstyle{#1}}%
   {\XXint\scriptscriptstyle\scriptscriptstyle{#1}}%
   \!\int}
\def\XXint#1#2#3{{\setbox0=\hbox{$#1{#2#3}{\int}$}
     \vcenter{\hbox{$#2#3$}}\kern-.5\wd0}}

\def\av_#1{\Xint-_{#1}}

\usepackage[textmath,displaymath,floats,graphics,delayed]{preview}
\title[]{On the inviscid   Boussinesq system with rough initial data}
\author[]{ZINEB HASSAINIA}
\address{IRMAR, Universit\'e de Rennes 1 \\ Campus de Beaulieu \\  35 042 Rennes cedex, France}
 \email{zineb.hassainia@univ-rennes1.fr}
 \author[T. Hmidi]{Taoufik Hmidi}
\address{IRMAR, Universit\'e de Rennes 1\\ Campus de
Beaulieu\\ 35~042 Rennes cedex\\ France}
\email{thmidi@univ-rennes1.fr}
\begin{document}
\subjclass[2000]{35Q35, 76B03, 76C05}
\keywords{ 2D inviscid Boussinesq system, vortex patches, local well-posedness}
\begin{abstract}
We deal with the local well-posedness theory  for the  two-dimensional inviscid Boussinesq system with rough initial   data of Yudovich type. The problem is in some sense  critical  due to some terms involving  Riesz transforms in the vorticity-density formulation.   We give a positive answer    for a special sub-class of Yudovich data including smooth and  singular vortex patches. For the latter case we assume in addition that the initial density is  constant around the singular part of the  patch boundary. 
\end{abstract}
\maketitle
\tableofcontents
\section{Introduction}
\quad\quad We consider the inviscid Boussinesq system describing the planar motion of a perfect incompressible  fluid evolving  under an external vertical force whose amplitude is proportional to the density which  is in turn  transported by the flow associated to  the velocity field. The corresponding  equations are given by,\begin{equation}\label{B}
\left\{ \begin{array}{ll}
\partial_{t}v+v\cdot\nabla v+\nabla p =\rho\vec{e}_2,\quad  t\geq0, x\in \RR^2 &\\
\partial_{t}\rho+v\cdot\nabla\rho =0, &\\
\textnormal{div}\, v=0, &\\
v_{\vert t=0}=v_0,\quad \rho_{\vert t=0}=\rho_0 .
\end{array} \right. 
\end{equation} 
Here the vector field $v = (v_1, v_2)$ and the scalar function $p$ denote   the fluid velocity  and the pressure  respectively.    The density  $\rho$ is a  passive scalar quantity  and the buoyancy force $\rho \vec{e_2}$ in the velocity equation models the gravity effect  on the fluid motion, where $\vec{e_2}$ stands for the unit vertical \mbox{vector $(0, 1)$.}
 
\quad This   system serves as  a simplified model  for the  fluid dynamics of the oceans and atmosphere. It  takes into account the stratification which plays a dominant role for large scales.   For more details about this subject see for instance \cite{B-C-R} and \cite{P}. 
The derivation of the  above system  can be formally  done  from the density dependent  Euler equations through the Oberbeck-Boussinesq approximation where the density fluctuation  is neglected everywhere in the momentum equation except  in the buoyancy force. We point out that Feireisl and Novotn\'y provide in \cite{F} a rigorous justification of the viscous  model  by means of scale analysis and  singular limit of   the full compressible  Navier-Stokes-Fourier system. 
  Note that the system \eqref{B} coincides with the classical  incompressible Euler equations when the initial density $\rho_0$ is identically  constant.   Recall that Euler system is given by,
\begin{equation}\label{E}
\left\{ \begin{array}{ll}
\partial_{t}v+v\cdot\nabla v+\nabla p =0, &\\
\textnormal{div}\, v=0, &\\
v_{\vert t=0}=v_0.
\end{array} \right. 
\end{equation} 
Before discussing some theoretical results on the well-posedness problem for the inviscid  Boussinesq equations we shall first start with the state of the art for the system \eqref{E}. The local existence and uniqueness of very smooth solutions  for \eqref{E} goes back to Wolibner \cite{W}  in the thirties of the last century. This result has been improved through the years by numerous authors and for several functional spaces. The pioneering work in this field  is  accomplished by  Kato and Ponce  in \cite{K-P} who proved the local well-posedness  in the framework of    Sobolev spaces $H^{s}$, with $s>\frac{d}{2}+1$.  This result was later generalized for other spaces, see for instance  \cite{C0,C,P-P,V,Z} and the references therein. Whether or not classical solutions develop singularities in finite time   is still  open except  some special cases as the planar motion  or the axisymmetric flows without swirl.  Unlike the viscous models, the global theory  for Euler equations has a geometric feature and    relies crucially on the vorticity dynamics. Historically, the concept of the vorticity   $\omega\triangleq \textnormal{rot}v$ and their basic laws  were studied by Helmholtz in his seminal work  on the vortex motion theory \cite{He}. More recently,  a blow-up vorticity criterion  for Kato's solutions  was  given by Beale, Kato and Majda in \cite{B-K-M}: the lifespan $T^\star$ is finite if and only if  $ \displaystyle{\int_0^{T^\star}\Vert \omega(\tau)\Vert_{L^\infty}d\tau=+\infty}.$
In two  dimensions the vorticity  can be identified to the scalar function $\omega=\partial_1 v_2-\partial_2 v_1$ and it  is transported by the flow, 
\begin{equation}\label{omegaeuler}
 \partial_t \omega+v\cdot\nabla\omega=0. 
\end{equation}
This leads to an  infinite family of conservation laws. For example, we have $\Vert \omega(t)\Vert_{L^p}= \Vert \omega_0\Vert_{L^p}$ for any $p\in [1,\infty]$. Hence the global well-posedness of Kato's solution  follows from the Beale-Kato-Majda criterion.

\quad By using the formal $L^p$ conservation laws it seems that we can relax the classical regularity and construct global weak solutions in $L^p$ spaces for $p > 1$.  This question has been originally addressed by Yudovich in \cite{Y}, where he proved the existence and uniqueness of weak solution to 2D Euler system only with the  assumption $\omega_0\in L^p\cap L^\infty$. Under this  pattern, the velocity is no longer  in the Lipschitz class  but belongs to the log-Lipschitz functions. With a velocity being in this latter class, it is proved that the flow map $\psi$ defined below is uniquely defined in the class of continuous  functions in both space and time variables, 
 $$
\psi(t,x)=x+\int_0^tv(\tau,\psi(\tau))d\tau.
$$
 We  can find more details about this subject in the book  \cite{C}. As a by-product we obtain  the global persistence of the vortex patch structure. More precisely, if the  initial vorticity  $\omega_0={\bf{1}}_{\Omega_0}$ is a patch, that is, the characteristic function of a bounded domain $\Omega_0$ then its evolution is given \mbox{$\omega(t)={\bf{1}}_{\Omega_t}$} with \mbox{$\Omega_t\triangleq\psi(t,\Omega_0)$.} The regularity persistence of the boundary is very subtle and was successfully accomplished by Chemin in \cite{C} who showed in particular  that when the boundary $\partial\Omega_0$ is better than $C^1$, say in   $C^{1+\varepsilon}$ for $0<\varepsilon<1$, then \mbox{$\partial\Omega_t$} will keep  its initial regularity for all the time  without any loss. The proof relies heavily on the estimate of the Lipschitz norm of the velocity with  the co-normal regularity $\partial_X\omega$ of the vorticity.  The  vector  fields $(X_t)$ are  transported by the flow, that is,
      \begin{equation}\label{tr71}
      \partial_t X+v\cdot\nabla X=X\cdot\nabla v.
      \end{equation}  
 The main advantage of this choice is the commutation of these vector fields   with the transport \mbox{operator $\partial_t+v\cdot\nabla $} which leads in turn  to the master equation
   \begin{equation}\label{transX}
  (\partial_t +v\cdot\nabla) \partial_X\omega=0. 
   \end{equation}
   This means that the tangential derivative of the vorticity is also transported by the flow and this is  crucial  in the framework of the vortex patches.   The proof  given by  Chemin is not restrictive to  the usual  patches but covers more singular  data called generalized vortex patches.  We point out that there is another  proof in the special case of the vortex patches  that can be found in  \cite{B-C}. It is also important to mention that Chemin  got in fact   more accurate result for patches with  singular boundary. In broad terms, he  showed that the regular part of the initial boundary  $\partial\Omega_0$ propagates with the same regularity without being affected by the singular part which   by the reversibility of the problem cannot be smoothed out by the dynamics and becomes  better  than $C^1$  .   
Furthermore, the  velocity $v$ is Lipschitz far from the singular set and may undergo a blowup behavior near  this set with a rate bounded by the logarithm of the distance  from the singular set.  Many similar  studies have been subsequently implemented by numerous authors  for bounded domains or viscous flows, see for instance   \cite{DA,D0,Da,De,GP,H,H1,S} and the references therein. 

 \quad Now, bearing in mind that the system \eqref{B} is at a formal level a perturbation of the incompressible Euler equations, it is legitimate to see whether the known results for Euler equations work for the Boussinesq system as well.  The earliest mathematical studies of the Boussinesq system and its dissipative counterpart are relatively recent and a great deal of attention has been paid  to   the local/global well-posedness problem, see for instance \cite{Ab,Dan, dp, H-K-R, HR2,HZ,HZ1, Lar,MZ, Wu}. Hereafter, we shall primarily restrict the discussion to the inviscid model described by 
\eqref{B} and recall some known facts on the classical solutions. We stress that this system can be seen as a hyperbolic one and therefore the commutator theory developed by Kato can be applied in a straightforward way. This was done by  Chae and Nam in  \cite{CN} who proved the  local well-posedness when the initial data $(v_0,\rho_0)$ belong to the sub-critical Sobolev space $ H^s$ with $s>2$. A similar result was also established later by the same authors  \cite{CN2}  for initial data lying in H\"olderian spaces $C^r$ with $r>1$. Another local  well-posedeness result is recently obtained in \cite{L2} for  the critical  Besov spaces $B^{2/p+1}_{p,1}$, with $p\in]1,+\infty[$. Furthermore, an analogous Beale-Kato-Majda criterion can be stated for the sub-critical cases.  More precisely, it can be shown that Kato's solutions cease to exist in finite time $T^\star$ if and only if $$
\int_0^{T^\star}\Vert \nabla\rho(t)\Vert_{L^\infty}dt=+\infty.
$$
For more details  we refer the reader for instance to \cite{L2,T}. Whether or not $T^\star$ is finite remains an outstanding open problem.    

\quad The main scope of  this paper is to deal with the  local well-posedness for \eqref{B} when the initial data are rough and belong  to   Yudovich class.  Contrary to the incompressible Euler equations the problem sounds extremely  hard to solve for generic Yudovich data due to the violent coupling between the vorticity and the density.  The difficulties can be illustrated from  the vorticity-density formulation,
\begin{equation}\label{omega}
\left\{ \begin{array}{ll}
\partial_{t}\omega+v\cdot\nabla\omega =\partial_1\rho, &\\
\partial_{t}\rho+v\cdot\nabla\rho =0, &
\end{array} \right. 
\end{equation}
According to the first equation in the above system one gets
$$
\|\omega(t)\|_{L^\infty}\le \|\omega_0\|_{L^\infty}+\int_0^t\|\nabla \rho(\tau)\|_{L^\infty} d\tau.
$$ 
 As we shall now see  the estimate of the last integral  requires the initial data to be more strong than what is allowed by Yudovich class. Indeed,  the partial derivative $\partial_j\rho $ obeys to the following transport model,
\begin{equation}\label{nablarho}
(\partial_{t}+v\cdot\nabla)\partial_j\rho =\partial_j v\cdot\nabla\rho.
\end{equation}
Consequently, the estimate of  $\|\nabla\rho(t)\|_{L^\infty} $  requires the velocity field to be at least  Lipschitz with respect to the space variable and  unfortunately this is  not necessary satisfied with a bounded vorticity. 
The main goal of   this paper is to give a positive answer for the local well-posedness problem for a special class of Yudovich data. We shall in the first part prove the  result for vortex patches with smooth boundary. In the second part we  conduct  the same study for patches with singular boundaries. 
 Our first result reads as follows
\begin{theorem}\label{the}
Let $0<\varepsilon<1$ and  $\Omega_0$ be a bounded domain  of the plane with a boundary $\partial\Omega_0$ in  H\"{o}lder class $C^{1+\varepsilon}$.
 Let $v_0$  be a divergence-free vector field of vorticity $\omega_0=1_{\Omega_0}$ and consider $\rho_0\in L^2\cap C^{1+\varepsilon}$ a real-valued function with  $\nabla\rho_0\in L^a$ and $1<a<2$. Then, there exists $T>0$ such that the Boussinesq  system  \eqref{B} admits a unique local solution $v,\rho\in L^\infty\big([0,T], W^{1,\infty}\big)$. Moreover, for all $t\in[0,T]$ the boundary of the advected domain   $\Omega_t=\psi(t,\Omega_0)$ is of class $C^{1+\varepsilon}$.
 \end{theorem}
Before giving some details about the proof  we shall discuss few remarks.
\begin{remark}
 The result of Theorem \ref{the} will be extended in Theorem $\ref{th1}$  to more general vortex structures. We shall get in particular a lower bound for the lifespan which is infinite for constant densities corresponding to the global result for Euler equations.  More precisely we get    $$
T^\star\geq \frac{1}{C_0}\log\bigg(1+C_0\log\Big(1+C_0/\Vert\nabla\rho_0\Vert_{L^\infty}\Big)\bigg).
$$
where $C_0\triangleq C_0(\omega_0,\rho_0)$ depends continuously on the involved norms.

\end{remark}
\begin{remark}
 For the sake of a clear presentation we have assumed in Theorem $\ref{the}$  that the density $\rho_0\in C^{1+\varepsilon}.$  The  persistence of such regularity is not clear and requires more than the Lipschitz norm for the velocity. However, as we shall see in Theorem $\ref{th1}$ this condition can be relaxed to   one that can be transported without loss: we replace this space by an anisotropic one.
 \end{remark}
\quad The proof of Theorem \ref{the} is firmly based on the formalism of vortex patches developed by Chemin  in \cite{C, C1}. The key is to estimate the  tangential regularity $\partial_X \omega$ in the H\"{o}lder space of negative index  $C^{\varepsilon-1}$, with respect to a suitable family of vector fields. Since this family commutes with the transport operator $\partial_t+v\cdot\nabla$ one gets easily the equation
\begin{eqnarray*}
(\partial_t+v\cdot\nabla)\partial_X\omega&=&\partial_X\partial_1\rho\\
&=&\partial_1(\partial_X\rho)+[\partial_X,\partial_1]\rho.
\end{eqnarray*}
By using para-differential calculus we can show that the commutator term is well-behaved and therefore the problem reduces  to the estimate $\|\partial_X\rho\|_{C^{\varepsilon}}.$ For this latter term we use anew the commutation  between $\partial_X$ and the transport operator combined with the fact that the density is also conserved along the particle  trajectories. Hence we find the equation
$$
(\partial_t+v\cdot\nabla)\partial_X\rho=0.
$$
This structure is very important in our analysis in order to derive some crucial a priori estimates. 

\quad Let us  move on  to the second contribution of this paper which is concerned with the singular vortex patches. We shall  assume that $\omega_0={\bf{1}}_{\Omega_0}$ but the  boundary may now contain a singular subset.  As the example of the square indicates,  the velocity associated to a vortex patch  is not in general  Lipschitz and this will bring more technical difficulties. Similarly to the smooth boundary one needs to bound $\|\nabla \rho(t)\|_{L^\infty}$ and from the characteristic method we obtain
$$
\Vert\partial_j\rho(t)\Vert_{L^\infty}\leq\Vert\nabla\rho_0\Vert_{L^\infty}+\int_0^t\Vert\partial_j v\cdot\nabla\rho(t)\Vert_{L^\infty}.
$$ 
We expect the singularities initially  located at the  boundary to be frozen in  the particle  trajectories and the idea to treat the last integral term  is to annihilate the effects of the velocity singularities by some specific assumptions on  the density.   As a possible choice we shall  assume  the  initial density to be constant around the singularity set and from its transport structure the density will remain constant around the image by the flow of the singular set. This allows to track the singularities and kill their nasty effects by the density.    
Our result reads as follows,
\begin{theorem}\label{the2}
Let $0<\varepsilon<1$ and $\Omega_0$ be a bounded domain of the plane whose boundary $\partial\Omega_0$ is a curve of class $C^{1+\varepsilon}$ outside a closed set $\Sigma_0$.
 Let us consider  a divergence-free vector field $v_0$ of vorticity  $\omega_0=1_{\Omega_0}$ and  take $\rho_0\in L^2\cap C^{\varepsilon+1}$  with   $\nabla\rho_0\in L^a$ for some  $1<a<2$. 
Suppose that $\rho_0$ is constant in a small neighborhood of  $\Sigma_0$. 
 Then the  system  \eqref{B} admits a unique local solution $(\omega,\rho)$ such that 
 $$\omega,\rho \in L^\infty\big([0,T], L^2\cap L^\infty ),\quad \nabla\rho\in L^\infty\big([0,T], L^a\cap L^\infty\big).
 $$  
Furthermore, the velocity $v$ is Lipschitz outside $\Sigma_t\triangleq \psi(t,\Sigma_0)$. More precisely, we have
$$
\sup_{h\in(0,e^{-1}]}\frac{\Vert \nabla v(t)\Vert_{L^\infty((\Sigma_t)_h^c)}}{-\log h}\in L^\infty([0,T]),
$$
where the set $(\Sigma_t)_h^c$ is defined by,
$$(\Sigma_t)_h^c\triangleq \big\{x\in \RR^2;\,d\big(x,\Sigma(t)\big)\geq h\big\}
.$$ In addition, the boundary of $\psi(t,\Omega_0)
$  
 is locally in $C^{1+\varepsilon}$ outside the set $\Sigma_t$.  
\end{theorem}
\begin{remark}
Let us mention that the  initial singular set is not arbitrary and should satisfy  a weak condition of the following type: there exists two strictly positive real numbers $\tilde{\gamma}$ and $C$ and a neighborhood $V_0$  of $\partial\Omega_0$ such that for any point $x\in V_0$ we have
$$\vert \nabla f(x)\vert\geq C d(x,\Sigma_0)^{\tilde{\gamma}}.$$
Here the function $f:\RR^2\to\RR$ is smooth and satisfies 
$$
\Omega_0=\big\{x, f(x)>0\big\}, \partial\Omega_0=\{x\in \RR^2, f(x)=0\}.
$$This means that the curves defining the boundary of $\Omega_0$ are not tangent to one another at infinite order at the singular points.
\end{remark}
\quad The general outline of the paper is as follows. In the next section  we recall some function spaces and  give some of their useful properties, we also gather  some preliminary estimates. Section 3 is devoted to the study of the regular vortex patches and the last section concerns the singular case. We close this paper with an appendix covering the proof of a technical lemma.

\section{Tools and function spaces}
\quad Throughout this paper, $C$ stands for some real positive constant which may be different in each occurrence and $C_0$ for a positive constant depending on the size of the  initial data. We shall sometimes alternatively use the notation $X\lesssim Y$ for an inequality of the type $X \leq CY$.
 
\quad Let us start with the dyadic partition of the unity whose  proof can be found for instance in  \cite{C}. There exists a radially symmetric function $\varphi$  in $\mathcal{D}(\RR^2\backslash\{0\})$ such that
\begin{equation*}\label{1}
\forall\xi\in\RR^2\backslash\{0\},\quad \sum_{q\in\mathbf{Z}}\varphi(2^{-q}\xi)=1.
\end{equation*}
We define the function $\chi\in\mathcal{D}(\RR^2)$ by
 \begin{equation*}\label{2}
\forall\xi\in\RR^2,\quad \chi(\xi)=1-\sum_{q\ge0}\varphi(2^{-q}\xi).
\end{equation*}
For every $u \in\mathcal{ S}'(\RR^2)$ one defines the non homogeneous Littlewood-Paley operators by,
$$
\Delta_{-1}v=\mathcal{F}^{-1}\big(\chi \hat{v}\big),\quad \forall{q}\in \NN \quad \Delta_{q} v=\mathcal{F}^{-1}\big(\varphi(2^{-q}\cdot) \hat{v}\big) \quad \textnormal{and}\quad S_{q}v=\sum_{-1\le j\le q-1}\Delta_{j}v.
$$
We notice that these operators map continuously  $L^p$ to itself uniformly with respect to $q$ and $p$. Furthermore, one can easily check that for every tempered distribution $v$, we have
	$$
	v=\sum_{q\geq -1}\Delta_qv.
	$$
By choosing in a suitable way the support of $\varphi$ one can easily check the  almost orthogonality properties: for any $u,v\in\mathcal{S}'(\RR^2)$,
$$
\Delta_p\Delta_q u=0\quad\textnormal{if}\quad\vert p-q\vert\geq 2
$$
$$
\Delta_p(S_{q-1}u\Delta_qv)=0\quad\textnormal{if}\quad\vert p-q\vert\geq 5.
$$
We can now give a characterization of the H\"older spaces using the Littlewood-Paley  decomposition. 
\begin{definition}
For all  $s \in \RR$, we denote by  $C^s$  the space of tempered distributions $v$ such that
$$
\Vert v\Vert_{s}\triangleq\sup_{q\geq -1}2^{qs}\Vert\Delta_{q}v\Vert_{L^\infty}<+\infty.
$$
\end{definition}
\begin{remark}
We notice that  for any strictly positive non integer real number $s$ this definition  coincides with the usual   H\"{o}lder space $C^s$ with equivalent norms.
 For example if $s\in]0,1[$,
$$
\Vert v\Vert_{s}\lesssim \Vert v\Vert_{C^s}\triangleq\Vert v\Vert_{L^\infty}+\sup_{x\neq y}\frac{\vert v(x)-v(y)\vert}{\vert x-y\vert^s}\lesssim\Vert v\Vert_{s}.
$$  
\end{remark}
Next, we recall Bernstein inequalities, see for example \cite{C}.
\begin{lemma}\label{br}
There exists a constant $C > 0$ such that for all $q \in N, k \in \NN, 1\leq a\leq b\leq\infty$ and for every tempered
distribution $u$ we have
$$
\sup_{\vert\alpha\vert\leq k}\Vert\partial^\alpha S_q u\Vert_{L^b}\leq C^k2^{q(k+2(\frac{1}{a}-\frac{1}{b}))}\Vert S_qu\Vert_{L^a},
$$
$$
C^{-k}2^{qk}\Vert \dot{\Delta}_qu\Vert_{L^b}\leq\sup_{\vert\alpha\vert= k}\Vert\partial^\alpha \dot{\Delta}_qu\Vert_{L^b}\leq C^k2^{qk}\Vert \dot{\Delta}_qu\Vert_{L^b}.
$$
\end{lemma}
Now, we introduce the Bony's decomposition \cite{B} which is  the basic tool of the para-differential calculus. 
Formally the product of two tempered distributions  $u$ and $v$ is splitted  into three parts as follows:
$$
uv=T_uv+T_vu+R(u,v),
$$
where
$$
T_uv=\sum_qS_{q-1}u\Delta_qv\quad\textnormal{and}\quad R(u,v)=\sum_q\Delta_qu\tilde{\Delta}_qv,
$$
$$
\textnormal{with}\quad\tilde{\Delta}_q =\sum_{i=-1}^1\Delta_{q+i}.
$$

The following lemma clarifies the behavior of the paraproduct  operators in the H\"older spaces.
\begin{lemma}\label{pare}
Let $s$ be a  real number. If $s<0$ the bilinear operator $T$ is continuous from $L^\infty\times C^s$ in $C^{s}$ and from $C^s\times L^\infty $ in $C^{s}$. Moreover, we have 
$$
\Vert T_uv\Vert_{s}+\Vert T_vu\Vert_{s}\leq C\Vert u\Vert_{L^\infty}\Vert v\Vert_s.
$$
 If $s>0$ the remainder operator $R$ is  continuous  from $L^\infty\times C^s$ in $C^{s}$. Furthermore, we have
$$
\Vert R(u,v)\Vert_{s}\leq C\Vert u\Vert_{L^\infty}\Vert v\Vert_s.
$$
Where $C$ is a positive constant depending only on $s$. 
\end{lemma}
As a result, we have the following corollary.
\begin{coro}\label{ppxr0}
Let $\varepsilon\in]0,1[$, $X$ be a vector field belonging to $C^\varepsilon$ as well as its divergence and $f$  be a Lipschitz scalar function. Then for $j\in\{1,2\}$ we have
$$
\Vert (\partial_jX)\cdot\nabla f\Vert_{\varepsilon-1}\leq C\Vert \nabla f\Vert_{L^\infty}\big(\Vert \textnormal{div}X\Vert_\varepsilon+\Vert X\Vert_\varepsilon\big).
$$ 
\end{coro}
\begin{proof}
 In view of Bony's decomposition we write
\begin{eqnarray*}
 \Vert(\partial_jX)\cdot\nabla f\Vert_{\varepsilon-1}&\leq& \Vert T_{\partial_jX^i}\partial_if\Vert_{\varepsilon-1}+\Vert T_{\partial_if}\partial_jX^i\Vert_{\varepsilon-1}+\Vert R(\partial_jX^i,\partial_if)\Vert_{\varepsilon-1},
\end{eqnarray*}
where we have adopted in the right-hand side of the last inequality the Einstein summation convention for the index $i$.
Since $\varepsilon-1<0$ the previous lemma ensures that
$$
\Vert T_{\partial_jX^i}\partial_if\Vert_{ \varepsilon-1}+\Vert T_{\partial_if}\partial_jX^i\Vert_{ \varepsilon-1}\leq C \Vert\nabla f \Vert_{ L^{\infty}}\Vert X\Vert_{ \varepsilon}.
$$
For the remainder term we write
$$
R(\partial_jX^i,\partial_if)=\partial_jR(X^i,\partial^if)-\partial_i R(X^i,\partial_jf)+R(\textnormal{div}\, X,\partial_jf).
$$ 
Using once again  Lemma \ref{pare} we get
\begin{eqnarray*}
\Vert R(\partial_jX^i,\partial_if)\Vert_{\varepsilon-1}&\lesssim & \Vert R(X^i,\partial^if)\Vert_{\varepsilon}+\Vert R(X,\partial_jf)\Vert_{\varepsilon}+\Vert R(\textnormal{div}\, X,\partial_jf)\Vert_{\varepsilon}\\ &\lesssim & \Vert\nabla f \Vert_{ L^{\infty}}\Vert X\Vert_{\varepsilon}+\Vert\nabla f \Vert_{L^{\infty}}\Vert\textnormal{div}\, X\Vert_{ \varepsilon}.
\end{eqnarray*}
This concludes the proof of the corollary.
\end{proof}

In the  next section we will need  the following result dealing with the H\"olderian regularity  persistence for the transport equations. Its proof is given in page 66 from \cite{C}.
\begin{lemma}\label{lem1}
Let $v$ be a smooth  divergence-free vector field and let $r \in ]-1,1[$.
Let us consider $(f,g)$ a couple of functions belonging to $L^\infty_{loc}(\RR, C^r)\times L^1_{loc}(\RR, C^r)$ and such that
$$
\partial_t f+v\cdot\nabla f=g.
$$
Then we have
\begin{equation}
\Vert f(t)\Vert_{r}\lesssim  \Vert f(0)\Vert_{r}e^{C\int_0^t\Vert\nabla v(\tau)\Vert_{L^\infty}d\tau}+\int_0^t\quad\Vert g(\tau)\Vert_{r}e^{C\int_\tau^t\Vert\nabla v(\sigma)\Vert_{L^\infty}d\sigma}d\tau
\end{equation}
The constant $C$ depends only on $r$.
\end{lemma}

\quad Next, we notice that if  $v$ is  divergence-free and  decaying at infinity then it can be recovered from  its  vorticity $\omega\triangleq \textnormal{rot} v$ by means of the Biot-Savart law 
\begin{equation}\label{b-s}
v(x)=\frac{1}{2\pi}\int_{\RR^2}\frac{(x-y)^\perp}{\vert x-y\vert^2}\omega(y)dy.
\end{equation}
Now we briefly recall the  Calder\'on-Zygmund estimate that will be frequently used through this paper.
\begin{proposition} There exists a positive constant $C$ satisfying the following property. For any  smooth divergence-free vector field $v$ with vorticity $\omega\in L^p$ and   $p\in]1,\infty[$ one has
\begin{equation}\label{c-z}
\Vert \nabla v\Vert_{L^p}\leq C\dfrac{p^2}{p-1}\Vert \omega\Vert_{L^p}.
\end{equation}
\end{proposition} 
\quad
 In order to extend the results stated in the introduction to various geometries, we shall introduce some useful notations and definitions. Namely, we define
 the anisotropic Besov spaces with respect to slight smooth vector fields. This approach has been initially  developed by J.-Y. Chemin in \cite{C} in order to treat the vortex patch problem for the incompressible Euler system.
\begin{definition}\label{def1}
 Let $\Sigma$ be a closed set of the plane and $\varepsilon\in (0,1).$ Let 
  $X=(X_\lambda)_{\lambda\in \Lambda}$ be a family  of vector fields. We say that this family is admissible of class   $C^\varepsilon$ outside $\Sigma$ if and only if :
  \begin{enumerate}
  \item Regularity:
  $X_\lambda , \textnormal{div }X_\lambda\in C^\varepsilon$. 
 \item Non degenracy: 
$$
I(\Sigma,X)\triangleq \inf_{x\notin\Sigma}\sup_{\lambda\in\Lambda}\vert X_\lambda(x)\vert>0.
$$
\end{enumerate}
We set
$$
\tilde{\Vert} X_\lambda\Vert_{\varepsilon}\triangleq\Vert X_\lambda\Vert_{\varepsilon}+\Vert \textnormal{div}X_\lambda\Vert_{ \varepsilon-1},
$$
and 
$$
N_\varepsilon(\Sigma,X)\triangleq\sup_{\lambda\in\Lambda}\frac{\tilde{\Vert} X_\lambda\Vert_{\varepsilon}}{I(\Sigma,X)}.
$$
\end{definition}
For each element  $X_\lambda$ of the preceding family we define its action on bounded  real-valued functions $u$ in the weak sense as follows:
$$
\partial_{X_\lambda}u\triangleq \textnormal{div}(u\, X_\lambda)-u\, \textnormal{div}X_\lambda.
$$\begin{definition}\label{def2}
Let $\varepsilon\in (0,1),\,k\in\NN$  and $\Sigma$ be a closed set of the plane. 
 Consider a family of vector fields $X=(X_\lambda)_\lambda$  as in the Definition $\ref{def1}.$ We denote by $C^{\varepsilon+k}(\Sigma, X)$ the space of functions $u\in W^{k,\infty}$  such that  
$$
\sum_{\vert \alpha\vert\leq k} \Vert  \partial^\alpha u\Vert_{L^\infty}+\sup_{\lambda\in\Lambda}\Vert \partial_{X_\lambda}u\Vert_{\varepsilon+k-1}<+\infty,
$$
and we set
\begin{equation*}
\Vert u\Vert^{\varepsilon+k}_{\Sigma,X}\triangleq  N_\varepsilon(\Sigma,X)\sum_{\vert \alpha\vert\leq k} \Vert  \partial^\alpha u\Vert_{L^\infty}+\displaystyle{\sup_{\lambda\in\Lambda}}\, \frac{\Vert \pxl u\Vert_{\varepsilon+k-1}}{I(\Sigma,X)}.
\end{equation*}
\end{definition}
\begin{remark}
When $\Sigma$ is empty we will  merely say  that the set of vector fields $(X_\lambda)_{\lambda\in\Lambda}$ is admissible and to make the notation less cluttered we shall withdraw the symbol $\Sigma$ from the previous definitions. For example,  we use simply  $I(X)$ instead of $I(X,\Sigma)$and $\Vert .\Vert^{\varepsilon+k}_{X}$ instead of  $\Vert .\Vert^{\varepsilon+k}_{\Sigma,X}$. 
\end{remark}
\quad The next result deals with a logarithmic estimate established  in \cite{C} which is the main  key in the study of the generalized vortex patches.\begin{theorem}\label{propoo1}
There exists an absolute constant $C $ such that for any $a\in (1,\infty), \varepsilon\in (0,1)$ we have the following property. 
Let  $\Sigma$ be a closed set of the plane and $X$ be a  family  of vector fields as in Definition $\ref{def1}$. Consider a function $\omega \in C^\varepsilon (\Sigma,X)\cap L^a$. Let $v$ be the divergence-free vector field with vorticity $\omega$, then we get:
\begin{equation*}\label{t1}
\Vert\nabla v\Vert_{L^\infty(\Sigma)}\leq Ca\Vert\omega\Vert_{L^a}+\frac{C}{\varepsilon}\Vert\omega\Vert_{L^\infty}\log\Big(e+\frac{\Vert\omega\Vert^{\varepsilon}_{\Sigma,X}}{\Vert\omega\Vert_{L^\infty}}\Big).
\end{equation*}
\end{theorem}
\section{Smooth  patches}
\quad In this section we shall state a  local well-posedness result  for the system \eqref{B} with general initial data covering the result of  Theorem \ref{the}. 
 The main result of this section is the following.
\begin{theorem}\label{th1}
Let $0<\varepsilon<1,a\in (1,\infty)$ and    $X_0=(X_{0,\lambda})_{\lambda\in\Lambda}$ be an  admissible family  of vector fields of class $C^\varepsilon$. Let $v_0$ be  a divergence-free vector field  whose vorticity $\omega_0$ belongs to $L^{a}\cap C^{\varepsilon}(X_0)$ and  $\rho_0$ be a real-valued  function belonging to $L^2\cap C^{\varepsilon+1}(X_0)$ with $\nabla\rho_0\in L^a$.
  Then there exists $T>0$ such that  the inviscid  Boussinesq system \eqref{B} admits a unique solution $(v,\rho)\in L^\infty_{loc}\big([0,T], Lip(\RR^2)\big)\times L^\infty_{loc}\big([0,T], Lip(\RR^2)\cap L^2\big)$ such that $\omega\in L^\infty\big([0,T], L^a\cap L^\infty\big)$. Moreover, for all $t\in[0, T]$ the transported $X_t$ of $X_0$ by the flow $\psi$,  defined by
\begin{equation}\label{p1}
X_{t,\lambda}(x)\triangleq \big(\pxzl\psi(t)\big)(\psi^{-1}(t,x)),
\end{equation}
 is admissible of class $C^\varepsilon$ and
 $$
 \rho(t)\in C^{\varepsilon+1}(X_t)\quad \textnormal{and}\quad \omega(t)\in C^{\varepsilon}(X_t).
 $$
In addition, 
 $$
T\geq \frac{1}{C_0}\log\bigg(1+C_0\log\Big(1+C_0/\Vert\nabla\rho_0\Vert_{L^\infty}\Big)\bigg)\triangleq T_0
$$
where $C_0\triangleq C_0(\omega_0,\rho_0)$ depends continuously on the norms of the initial data
\end{theorem}
\begin{remark}
The Theorem \ref{th1} can be applied to  a  larger class of  initial data than the vortex patches class.  For  example, we may take $\omega_0=\tilde{\omega}_01_{\Omega_0}$ with $\Omega_0$ a bounded domain of class $C^{\varepsilon+1}$ and $\tilde{\omega}_0$ a function of class $C^\varepsilon(\RR^2)$ for some $\varepsilon\in]0,1[$. 
\end{remark}
We shall now make precise the boundary regularity used in the main theorems.
\begin{definition}\label{bord}
 Let $0<\varepsilon<1$ and  $\Omega$ be  a bounded domain in $\RR^d$. We say that $\Omega$ is of class  $C^{1+\varepsilon}$ if  there exists a compactly supported  function $f \in C^{1+\varepsilon}(\RR^2)$ and  a neighborhood $V$ of $\partial\Omega$ such that
 $$
\partial\Omega=  f^{-1}(\{0\})\cap V\quad\textnormal{and}\quad \nabla f(x)\neq  0\quad \forall x\in V.
$$

\end{definition}
Let us see how  to deduce the results of  Theorem \ref{the} from the preceding one.
\subsection{Proof of Theorem \ref{the}}

 To begin with,  we shall construct an admissible family of vector fields $X_0$ for which the initial vorticity $\omega_0=1_{\Omega_0}$ satisfies the tangential regularity property. 
 In view of the previous definition, there exists a real function $f_0\in C^{1+\varepsilon}$ and a neighborhood $V_0$ such that $\partial\Omega_0=V_0\cap f^{-1}(\{0\})$ and  $\nabla f_0\not\equiv 0$ on $V_0$. 
 Let $\tilde{\alpha}$ be a smooth function supported in $V_0$ and taking the value $1$ in a small neighborhood of $V_1\subset V_0$. We set
$$
X_{0,0}=\nabla^\perp f_0,\quad X_{0,1}=(1-\tilde{\alpha})\begin{pmatrix}1\\ 0\end{pmatrix}.
$$
The first vector field is of class $C^\varepsilon$ with zero divergence, the second is $C^\infty$ and a simple verification shows that the family of vector fields $(X_{0, i})_{i\in\{0,1\}}$ is admissible . Besides, since the derivative of $\omega_0$ along the direction $\nabla^\perp f_0$ is zero  and $1-\tilde{\alpha}$ vanishes on $V_1$ then we have $\partial_{X_{0, i}}\omega_0=0$.
 
Also, the fact that  $\rho_0\in C^{\varepsilon+1}$ implies that $\rho_0\in C^{\varepsilon}\big((X_{0, i})_{i\in\{1,2\}}\big)$. 
Therefore  Theorem \ref{the} provides  a unique local solution $(v,\rho)\in L^\infty_{loc}\big([0,T_0], Lip(\RR^2)\big)^2$ to \eqref{B}. 
 For the regularity of  the transported initial domain $\Omega_t=\psi(t,\Omega_0)$, we consider $\gamma^0\in C^{\varepsilon+1}(\RR_+,\RR^2)$ a parametrization of $\partial\Omega_0$ given by
\begin{equation*}
\left\{ \begin{array}{ll}
\partial_{\sigma}\gamma^0=\nabla^\perp f_0(\gamma^0(\sigma)), &\\
\gamma^0(0) =x_0\in\partial\Omega_0.&\\
\end{array} \right. 
\end{equation*} 
Set $\gamma_t(\sigma)=\psi(t,\gamma^0(\sigma))$, then by differentiating with respect to the parameter $\sigma$ we get
\begin{equation*}
\left\{ \begin{array}{ll}
\partial_{\sigma}\gamma_t(\sigma)=(\partial_{X_{0,0}}\psi)(t,\gamma^0(\sigma)), &\\
\gamma_t(0) =\psi(t,x_0).&\\
\end{array} \right. 
\end{equation*} 
From Theorem \ref{th1}, $\partial_{X_{0,0}}\psi$ belongs to $L^\infty_{loc}([0,T_0], C^\varepsilon)$, then $\gamma_t$ belongs to $L^\infty_{loc}(\RR_+, C^{\varepsilon+1})$ for all $t\leq T_0$. 
Finally, as $X_{0,0}$ does not vanish on $V_0$, then it is the same for $\partial_{X_{0,0}}\psi$, therefore, $\partial_{\sigma}\gamma_t$  does not vanish on $\RR$ as indicated by the estimate \eqref{pp2}. Consequently, $\gamma_t$  is a regular parameterization of $\partial\Omega_t$.

\subsection{A priori estimates}
This part is the core of the proof of Theorem \ref{th1}. As a matter of fact, we aim here to propagate the regularity of the initial data, namely, to  bound the norms $\Vert \omega(t)\Vert_{L^a\cap L^\infty}$ and $\Vert \nabla\rho(t)\Vert_{L^a\cap L^\infty}$. 
 Even though these quantities seem to be less regular than $\Vert \nabla v(t)\Vert_{L^1_tL^\infty}$ , it is not at all clear how to estimate them without involving the latter quantity.  It comes then to show the two following propositions: The first deals with the $L^p$ estimates and the second is related
on the estimate of the Lipschitz norm for the  solution of the system \ref{B}.
\begin{proposition}\label{lp}
Let $(v,\rho)$ be a smooth solution of the Boussinesq system \eqref{B} defined on the time interval $[0,T]$. Then, for all $p\in[1,+\infty]$ and $t\leq T$ we have 
\begin{eqnarray}\label{pp7}
\Vert\omega(t)\Vert_{L^p}&\leq & \Vert\omega_0\Vert_{L^p}+\Vert\nabla \rho_0\Vert_{L^p} e^{CV(t)}t.
\end{eqnarray}
and
\begin{equation}\label{rho}
\Vert\nabla \rho(t)\Vert_{L^p}\leq \Vert\nabla \rho_0\Vert_{L^p} e^{CV(t)}.
\end{equation}
with the notation:
$$
V(t)\triangleq\int_0^t\|\nabla v(\tau)\|_{L^\infty} d\tau.
$$
\end{proposition}
\begin{proof}
Using the vorticity  equation \eqref{omega} we can easily see that for all $1\leq p\leq\infty$,
\begin{eqnarray}\label{ppp6}
\Vert\omega(t)\Vert_{L^p}&\leq & \Vert\omega_0\Vert_{L^p}+\int_0^t\Vert\nabla\rho(\tau)\Vert_{L^p}d\tau.
\end{eqnarray}
Next, applying the partial derivative  operator $\partial_j$ to the second equation of the system \eqref{B}, we get
\begin{equation*}\label{eqprho}
\partial_t\partial_j \rho +v\cdot\nabla(\partial_j \rho)=\partial_j v\cdot \nabla\rho.
\end{equation*}
Hence, for all $1\leq p\leq \infty$, we obtain
\begin{eqnarray*}
\Vert\partial_j \rho(t)\Vert_{L^p}&\leq & \Vert\nabla \rho_0\Vert_{L^p}+\int_0^t\Vert\nabla\rho(\tau)\Vert_{L^p}\Vert\nabla v(\tau)\Vert_{L^\infty}d\tau.
\end{eqnarray*}
According to the Gronwall lemma we conclude that
\begin{equation*}
\Vert\nabla \rho(t)\Vert_{L^p}\leq \Vert\nabla \rho_0\Vert_{L^p} e^{CV(t)}.
\end{equation*}
Plugging this estimate into \eqref{ppp6} gives,
\begin{eqnarray*}
\Vert\omega(t)\Vert_{L^p}&\leq & \Vert\omega_0\Vert_{L^p}+\Vert\nabla \rho_0\Vert_{L^p} e^{CV(t)}t;\quad\forall 1\leq p\leq\infty.
\end{eqnarray*}
\end{proof}
Next we shall discuss  the Lipschitz norm  of the velocity. This parts uses the formalism of the vortex patches. Our result reads as follows.
\begin{proposition}\label{prop1}
Let  $0<\varepsilon<1$, $a>1$ and  $X_0$ be an  admissible family  of vector fields of class $C^\varepsilon$. Let  $(v,\rho)$ be a smooth solution of the Boussinesq system \eqref{B} defined on the time interval $[0,T^\star[$. Then there exists $0<T_0\leq T^\star $ such that for all time $t\leq T_0$   we have
\begin{equation*}\label{lip}
\Vert \nabla v(t)\Vert_{L^\infty}\leq  C_{0}.
\end{equation*}
 \end{proposition}
 The  proof of this proposition is firmly based on  the following lemma.
 \begin{lemma}
There exists a constant $C$ such that for any smooth solution $(v,\rho)$ of \eqref{B} on $[0,T]$, and any time dependent family of vector field $X_t$ transported by the flow of $v$, we have for all $t\in[0,T]$,
\begin{equation}\label{pp2}
I(X_t)\geq I(X_0)e^{-V(t)}.
\end{equation}
\begin{equation}\label{div}
\Vert\textnormal{div} X_{t,\lambda}\Vert_{\varepsilon}\leq \Vert\textnormal{div} X_{0,\lambda}\Vert_{\varepsilon}e^{CV(t)}.
\end{equation}
\begin{eqnarray}\label{xpx}
\Vert X_t\Vert_\varepsilon+\Vert \partial_{X_{t,\lambda}}\omega\Vert_{\varepsilon-1}&\leq &C\Big(\tilde{\Vert} X_{0,\lambda}\Vert_{ \varepsilon}+\Vert \pxzl\omega_0\Vert_{ \varepsilon-1}+\Vert \pxzl\rho_0\Vert_{\varepsilon}\Big)e^{Ct}e^{CV(t)}\notag\\ &\times &\exp\big( t\Vert \nabla \rho_0\Vert_{L^\infty}e^{CV(t)}\big).
\end{eqnarray}
\begin{eqnarray*}
\Vert \pxtl\rho\Vert_{ \varepsilon}&\leq & \Vert \pxzl\rho_0\Vert_{ \varepsilon}e^{CV(t)}.
\end{eqnarray*}
where
$$
V(t)\triangleq \int_0^t\Vert \nabla v(\tau)\Vert_{L^\infty}d\tau.
$$
 \end{lemma}
\begin{proof}
Taking the derivative of  term $\pxzl\psi(t,x)$ with  respect to the time $t$ we get
\begin{equation}\label{X0}
\left\{ \begin{array}{ll}
\partial_t \pxzl\psi(t,x)=\nabla v(t,\psi(t,x))\pxzl\psi(t,x). &\\
\pxzl\psi(0,x)=X_{0,\lambda}.
\end{array} \right. 
\end{equation} 
Using the time reversibility of this equation combined with  Gronwall's lemma  we find
\begin{equation*}\label{adm}
 \vert X_{0,\lambda}(x)\vert\leq \vert \pxzl\psi(t,x)\vert e^{V(t)}.
 \end{equation*}
From the Definition \ref{def1} and the relation \eqref{p1} we obtain the desired he estimate \eqref{pp2}.\\
It is easy to check from the relation (\ref{X0}) that 
\begin{equation}\label{Xt}
\partial_t X_{t,\lambda}+v\cdot  X_{t,\lambda}=\pxtl v.
\end{equation}
Applying the divergence operator to the equation  \eqref{Xt} we obtain,
\begin{equation*}
\big(\partial_t +v\cdot\nabla\big)\textnormal{div}X_{t,\lambda}=0,
\end{equation*}
and therefore  we may use   Lemma \ref{lem1} leading to the estimate \eqref{div}.\\
Next, we intend to establish \eqref{xpx}. For this goal we start  with the  following result whose proof is given in Lemma 3.3.2  of  \cite{C},
  \begin{eqnarray*}
\Vert \pxtl v(t)\Vert_{\varepsilon}&\lesssim & \Vert\nabla v(t)\Vert_{L^\infty}\tilde{\Vert} X_{t,\lambda}\Vert_{\varepsilon}+\Vert \pxtl\omega(t)\Vert_{ \varepsilon-1}.
\end{eqnarray*}
Applying  Lemma \ref{lem1} to equation (\ref{Xt}) we get
\begin{eqnarray*}
\Vert X_{t,\lambda}\Vert_{\varepsilon}&\leq & e^{CV(t)}\Big( \Vert X_{0,\lambda}\Vert_{\varepsilon}+C\int_0^t\big(\Vert\nabla v(\tau)\Vert_{L^\infty}\tilde{\Vert} X_{\tau,\lambda}\Vert_{\varepsilon}+\Vert \partial_{X_{\tau,\lambda}}\omega(\tau)\Vert_{ \varepsilon-1}\big)e^{-CV(\tau)}d\tau\Big).
\end{eqnarray*}
Putting this estimate with \eqref{div} yields
\begin{eqnarray}\label{pp33}
\tilde{\Vert} X_{t,\lambda}\Vert_{\varepsilon}&\leq & e^{CV(t)}\Big(\tilde{\Vert} X_{0,\lambda}\Vert_{\varepsilon}+ C\int_0^t\big(\Vert\nabla v(\tau)\Vert_{L^\infty}\tilde{\Vert} X_{\tau,\lambda}\Vert_{\varepsilon}+\Vert \partial_{X_{\tau,\lambda}}\omega(\tau)\Vert_{ \varepsilon-1}\big)e^{-CV(\tau)}d\tau\Big).
\end{eqnarray}
Since  $\pxtl$ commutes with the transport  operator $\partial_t +v\cdot\nabla$, then
 \begin{equation*}
\big(\partial_t +v\cdot\nabla\big)\pxtl\omega=\pxtl\partial_1\rho
\end{equation*}
and consequently we get in view of  Lemma \ref{lem1} 
\begin{eqnarray}\label{pp4}
\Vert\pxtl\omega(t)\Vert_{ \varepsilon-1}&\leq &e^{CV(t)}\Big(\Vert \pxzl\omega_0\Vert_{ \varepsilon-1}+ C\int_0^t\Vert  \partial_{X_{\tau,\lambda}}\partial_1\rho(\tau)\Vert_{\varepsilon-1}e^{-CV(\tau)}d\tau\Big).
\end{eqnarray}
Observe that
\begin{equation}
\partial_{X_{\tau,\lambda}}\partial_1\rho=\partial_1 (\partial_{X_{\tau,\lambda}}\rho)-\partial_{\partial_1X_{\tau,\lambda}}\rho,
\end{equation}
and thus
\begin{eqnarray}\label{pp5}
\Vert\partial_{X_{\tau,\lambda}}\partial_1\rho(\tau)\Vert_{\varepsilon-1}&\lesssim & \Vert \partial_{X_{\tau,\lambda}}\rho(\tau)\Vert_\varepsilon+\Vert(\partial_{1}X_{\tau,\lambda})\cdot\nabla\rho(\tau)\Vert_{\varepsilon-1}\notag \\  &\lesssim &\Vert \partial_{X_{\tau,\lambda}}\rho(\tau)\Vert_\varepsilon+\Vert\nabla\rho(\tau) \Vert_{ L^{\infty}}\tilde{\Vert} X_{\tau,\lambda}\Vert_{ \varepsilon},
\end{eqnarray} 
where we have used  in the last inequality  Corollary \ref{ppxr0}. \\
To estimate the term $\Vert \partial_{X_{\tau,\lambda}}\rho(\tau)\Vert_\varepsilon$ we use once again the commutation between   $\pxtl$ and the  transport  operator leading to,
\begin{equation}\label{r1}
\big(\partial_t +v\cdot\nabla\big)\pxtl\rho=0.
\end{equation}
Applying  Lemma \ref{lem1} gives
$$
\Vert\partial_{X_{t,\lambda}}\rho(t)\Vert_{\varepsilon}\lesssim  \Vert \partial_{X_{0,\lambda}}\rho_0\Vert_\varepsilon e^{CV(\tau)}
$$ 
which yields according to   \eqref{pp5} 
\begin{eqnarray*}
\Vert\partial_{X_{\tau,\lambda}}\partial_1\rho(\tau)\Vert_{\varepsilon-1}&\lesssim & \Vert \partial_{X_{0,\lambda}}\rho_0\Vert_\varepsilon e^{CV(\tau)}+\Vert\nabla\rho(\tau) \Vert_{ L^{\infty}}\tilde{\Vert} X_{\tau,\lambda}\Vert_{ \varepsilon}.
\end{eqnarray*}
Plugging this estimate into (\ref{pp4}) implies
\begin{eqnarray*}
\Vert\pxtl\omega(t)\Vert_{ \varepsilon-1}\lesssim e^{CV(t)}\Big(\Vert \pxzl\omega_0\Vert_{ \varepsilon-1}+\Vert \partial_{X_{0,\lambda}}\rho_0\Vert_{ \varepsilon}t+\int_0^t\Vert\nabla\rho(\tau) \Vert_{ L^{\infty}}\tilde{\Vert} X_{\tau,\lambda}\Vert_{ \varepsilon}e^{-CV(\tau)}d\tau\Big).
\end{eqnarray*}
Hence, putting together the foregoing  estimate  and  (\ref{pp33})  we get
\begin{equation*}
\Gamma(t)\lesssim \Gamma(0)+\Vert \partial_{X_{0,\lambda}}\rho_0\Vert_{ \varepsilon}t+\int_0^t \big(\Vert \nabla \rho\Vert_{L^\infty}+\Vert \nabla v\Vert_{L^\infty}+1\big)\Gamma(\tau)d\tau.
\end{equation*}
with $\Gamma(t)\triangleq \big(\Vert\pxtl\omega(t)\Vert_{ \varepsilon-1}+\tilde{\Vert} X_{t,\lambda}\Vert_{ \varepsilon}\big)e^{-CV(t)}$.
Then Gronwall's lemma implies  that
\begin{equation*}
\Gamma(t)\lesssim \Big(\Gamma(0)+\Vert \partial_{X_{0,\lambda}}\rho_0\Vert_{ \varepsilon}\Big)e^{C\int_0^t\big(\Vert \nabla \rho\Vert_{L^\infty}+\Vert \nabla v\Vert_{L^\infty}+1\big)d\tau}.
\end{equation*}
Finally, Proposition \ref{lp} gives the desired result.
\end{proof}
\vspace{0,2cm}

\textit{Proof of the Proposition $\ref{prop1}$. } 
Combining the inequality \eqref{pp7} with the estimate \eqref{xpx}  we find
\begin{eqnarray*}
\Vert \pxtl\omega(t)\Vert_{ \varepsilon-1}+\Vert \omega(t)\Vert_{L^\infty}\tilde{\Vert} X_{t,\lambda}\Vert_{\varepsilon}\leq \big(1+\Vert\omega_0\Vert_{L^\infty}\big)\big(\Gamma(0)+\Vert \partial_{X_{0,\lambda}}\rho_0\Vert_{ \varepsilon}\big)e^{ CV(t)}e^{Ct}e^{Ct\Vert \nabla \rho_0\Vert_{L^\infty}e^{CV(t)}}.
\end{eqnarray*}
Putting together the last estimate and  the inequality \eqref{pp2} then we get according to  the  Definition \ref{def2}
 \begin{eqnarray}\label{inter}
 \Vert \omega(t)\Vert^{\varepsilon}_{X_t} \leq C_0 e^{CV(t)}e^{Ct}e^{Ct\Vert \nabla\rho_0\Vert_{L^a\cap L^\infty}e^{CV(t)})},
\end{eqnarray}
According to the Proposition \ref{propoo1} and the monotonicity  of the map $x\longmapsto x\log \big(e + \frac{a}{x}\big)$ we find
\begin{eqnarray*}
\Vert\nabla v(t)\Vert_{L^\infty}\leq  C\Big(\Vert \omega_0\Vert_{L^a\cap L^\infty}+t\Vert \nabla\rho_0\Vert_{L^a\cap L^\infty}e^{CV(t)}\Big)\log\bigg(e+\frac{ \Vert \omega(t)\Vert^{\varepsilon}_{X_t}}{\Vert \omega_0\Vert_{L^\infty}}\bigg).
\end{eqnarray*}
It follows from the estimate \eqref{inter} that
\begin{eqnarray}\label{121}
\Vert\nabla v(t)\Vert_{L^\infty}&\leq & C\Big(\Vert \omega_0\Vert_{L^a\cap L^\infty}+t\Vert \nabla\rho_0\Vert_{L^a\cap L^\infty}e^{CV(t)}\Big)\notag\\ &\times &\bigg(C_0+t+t\Vert \nabla\rho_0\Vert_{L^a\cap L^\infty}e^{CV(t)}+V(t) \bigg),\end{eqnarray}
We shall take $T>0$ such that 
\begin{equation}\label{time}
T\Vert\nabla\rho_0\Vert_{L^a\cap L^\infty}e^{CV(T)}\leq  \textnormal{min}\big(1,\Vert\omega_0\Vert_{L^1\cap L^\infty}\big).
\end{equation}
Then we deuce from  \eqref{121}
$$
\Vert\nabla v(t)\Vert_{L^\infty}\leq  C\Vert \omega_0\Vert_{L^a\cap L^\infty}\Big(C_0+t+\int_0^t\|\nabla v(\tau)\|_{L^\infty}d\tau\Big),\quad \forall t\in [0,T].
$$ 
which yields in view of  Gronwall lemma 
\begin{eqnarray*}
\Vert\nabla v(t)\Vert_{L^\infty}&\leq & C\Vert \omega_0\Vert_{L^a\cap L^\infty}(C_{0}+t)e^{C\Vert \omega_0\Vert_{L^a\cap L^\infty}t}\quad\textnormal{for all}\quad t\in[0,T],
\end{eqnarray*}
Therefore in order to satisfy the assumption \eqref{time}, it suffices that
$$
T\Vert\nabla\rho_0\Vert_{L^a\cap L^\infty}\exp\Big((C_{0}+T)\big(e^{C\Vert \omega_0\Vert_{L^a\cap L^\infty}T}-1\big)\Big)\leq  \textnormal{min}\big(1,\Vert\omega_0\Vert_{L^1\cap L^\infty}\big).
$$
Hence, a possible choice for $T$ is given by the formula
\begin{equation}\label{lif1}
T\triangleq \frac{1}{C\Vert \omega_0\Vert_{L^a\cap L^\infty}}\log\bigg(1+\frac{\Vert \omega_0\Vert_{L^a\cap L^\infty}}{\Vert \omega_0\Vert_{L^a\cap L^\infty}C_{0}+1}\log\Big(1+\frac{C\textnormal{min}\big(\Vert \omega_0\Vert_{L^a\cap L^\infty},\Vert \omega_0\Vert_{L^a\cap L^\infty}^2\big)}{\Vert\nabla\rho_0\Vert_{L^\infty}}\Big)\bigg).
\end{equation}
\subsection{Existence}
The main goal of this paragraph is to answer to the  local existence part mentioned in  Theorem \ref{the}. For this aim we shall  we consider the following system
\begin{equation*}
\left\{ \begin{array}{lll}
\partial_{t}v_n+v_n\cdot\nabla v_n+\nabla p_n =\rho_n\vec{e}_2, &\\
\partial_{t}\rho_n+v_n\cdot\nabla\rho_n =0, &\\
\textnormal{div}\, v_n=0,\\
v_{0,n} = S_nv_0,\quad \rho_{0,n} = S_n\rho_0.
\end{array} \right. 
\end{equation*}

where $S_n$ is the usual  cut-off in frequency defined in Section 2. Since the initial data $v_{0,n}, \rho_{0,n} $ are smooth and belong to $C^s, s>1$ then we can 
apply  Chae's  result \cite{CN2} and get  for each $n$ a unique local solution $v_n,\rho_n\in C\big([0,T^*_n[,C^s\big).$
 The   maximal time existence  $T_n^*$ obeys to   the following blow-up  criterion.
\begin{equation}\label{bc}
T_n^\star<\infty\Longrightarrow \int_0^{T_n^*}\Vert \nabla v_n(\tau)\Vert_{L^\infty}d\tau=+\infty.
\end{equation}
To get a uniform time existence, that is, $\liminf_{n\to\infty}T_n^\star>0$ it suffices to check that the time existence $ \liminf_{n\to\infty}T_n\geq T,$ where $T$ is given by \eqref{lif1} and   $T_n$ is defined by  \eqref{lif1} with  the smooth data. To do so, it suffices first  to check the uniformness of the constant depending on the size of the initial data and we shall see second how to achieve the argument. First, we should bound uniformly the quantities 
$$\Vert\omega_{0,n}\Vert_{L^a \cap L^\infty},
\Vert\nabla \rho_{0,n}\Vert_{L^a \cap L^\infty},\Vert \omega_{0,n}\Vert^{\varepsilon}_{X_{0}},\Vert \rho_{0,n}\Vert^{\varepsilon+1}_{X_{0}}.
$$
This follows from the uniform continuity of the operator $S_n:L^p\to L^p$   and by the following estimates  stated in pages 62, 63 from \cite{C}
$$
\Vert \partial_{X_{0,\lambda}}\omega_{0,n}\Vert_{\varepsilon-1}\leq C\big(\Vert \partial_{X_{0,\lambda}}\omega_{0}\Vert_{\varepsilon-1}+\tilde{\Vert} X_{0,\lambda}\Vert_{\varepsilon}\Vert\omega_0\Vert_{L^\infty}\big).
$$
By the same way we may prove that 
$$
\Vert \partial_{X_{0,\lambda}}\rho_{0,n}\Vert_{\varepsilon}\leq C\big(\Vert \partial_{X_{0,\lambda}}\rho_0\Vert_{\varepsilon}+\tilde{\Vert} X_{0,\lambda}\Vert_{\varepsilon}\Vert\nabla\rho_0\Vert_{L^\infty}\big).
$$
To complete the proof of the claim, we assume that for some $n$ we have $T_n^\star\le T_0$ where $T_0$ is given by \eqref{lif1} , then all the a priori estimates done in the preceding section are justified and therefore we obtain 
 according to the Proposition \ref{prop1}
\begin{eqnarray*}
\Vert\nabla v_n(t)\Vert_{L^\infty}&\leq &C_{0},
\end{eqnarray*}
\begin{equation*}
 \Vert\omega_n(t)\Vert_{L^a \cap L^\infty}+\Vert\nabla \rho_n(t)\Vert_{L^a \cap L^\infty}\leq C_{0}
\end{equation*}
and
\begin{eqnarray*}
\Vert \rho_n(t)\Vert^{\varepsilon+1}_{X_{t,n}}+\Vert \omega_n(t)\Vert_{X_{t,n}}^{\varepsilon}++\sup_{\lambda\in \Lambda}\Vert \partial_{X_{0,\lambda}}\psi_n(t)\Vert_{\varepsilon}\leq C_{0}.
\end{eqnarray*}
Where $\psi_n$ is the flow associated to the vector field $v_n$. This contradicts the blow-up criterion \eqref{bc} and consequently $T^*_n> T_0$.
By standard compactness arguments we can show that this family $(v_n,\rho_n)_{n\in\NN}$ converges to $(v,\rho)$ which satisfies  our initial value problem. We omit here the details and we will next focus on the uniqueness part. 
\subsection{Uniqueness}
We shall  now focus on the  uniqueness part which will be performed in the functions space  $\mathcal{X}_{T_0}=L^\infty([0,T_0],L^q\cap W^{1,\infty})$for some  $2<q<\infty$. We point out that this space is larger than the space of the existence part and the restriction to $q>2$ comes from the fact that the velocity associated to a vortex patch is not in $L^2$, due to its slow decay at infinity, but belongs to the spaces $L^q,\forall q>2.$
Let $(v_1,p_1,\rho_1)$ and $(v_2,p_2,\rho_2)$ be two solutions of the system \eqref{B} belonging to the space $\mathcal{X}_{T_0}$ and let us denote by
$$
v=v_1-v_2,\quad p=p_1-p_2\quad\textnormal{and}\quad \rho=\rho_1-\rho_2.
$$
Then we have the system
\begin{equation*}
\left\{ \begin{array}{ll}
\partial_{t}v+v_2\cdot\nabla v=-v\cdot\nabla v_1-\nabla p +\rho\vec{e}_2, &\\
\partial_{t}\rho+v_2\cdot\nabla\rho =-v\cdot\nabla \rho_1, &\\
 v_{\vert_{t=0}}=v_0,\quad \rho_{\vert_{t=0}}=\rho_0.
\end{array} \right. 
\end{equation*} 
The $L^q$ estimate of the density is given by
\begin{equation}\label{rlq}
\Vert \rho(t)\Vert_{L^q}\leq \Vert \rho_0\Vert_{L^q}+\int_0^t\Vert v(\tau)\Vert_{L^q}\Vert \nabla \rho_1\Vert_{L^\infty}d\tau.
\end{equation}
Similarly we estimate  the velocity as follows,
\begin{equation}\label{vlq}
\Vert v(t)\Vert_{L^q}\leq \Vert v_0\Vert_{L^q}+\int_0^t\big(\Vert v(\tau)\Vert_{L^q}\Vert \nabla v_1\Vert_{L^\infty}+\Vert \nabla p(\tau)\Vert_{L^q}+\Vert \rho(\tau)\Vert_{L^q}\big)d\tau.
\end{equation}

But using the incompressibility condition we get
\begin{eqnarray*}
\nabla p &=&\nabla\Delta^{-1}\textnormal{div}\big(-v\cdot\nabla v_1+\rho \vec{e_2}\big)-\nabla\Delta^{-1}\textnormal{div}(v_2\cdot\nabla v)\\ &=& \nabla\Delta^{-1}\textnormal{div}\big(-v\cdot\nabla( v_1+v_2)+\rho e_2\big) .
\end{eqnarray*}
where we have used in the last equality the fact that $\textnormal{div}(v_2\cdot\nabla v)=\textnormal{div}(v\cdot\nabla v_2)$. By the continuity of Riesz  transform on $L^q$ we obtain
$$
\Vert\nabla p\Vert_{L^q}\leq C\Big(\Vert v\Vert_{L^q}\big(\Vert \nabla v_1\Vert_{L^\infty}+\Vert \nabla v_2\Vert_{L^\infty}\big)+\Vert \rho\Vert_{L^q}\Big).
$$
Inserting the last estimate into \eqref{vlq} and using the continuity of Riesz transforms  one gets
\begin{equation*}
\Vert v(t)\Vert_{L^q}\leq \Vert v_0\Vert_{L^q}+C\int_0^t\Big(\Vert v(\tau)\Vert_{L^q}\big(\Vert \nabla v_1(\tau)\Vert_{L^\infty}+\Vert \nabla v_2(\tau)\Vert_{L^\infty}\big)+\Vert\rho\Vert_{L^q}\Big)d\tau.
\end{equation*}
Combining the last estimate with \eqref{rlq}  and using   Gronwall inequality we find that for all $t\leq T_0$ we have
\begin{eqnarray*}
\Vert (v(t),\rho(t))\Vert_{L^q}&\leq& \Vert (v_0,\rho_0)\Vert_{L^q}e^{Ct}\exp\Big(\int_0^t \big(\Vert \nabla v_1\Vert_{L^\infty}+\Vert \nabla v_2\Vert_{L^\infty}+\Vert \nabla \rho_1\Vert_{L^\infty}\big)d\tau\Big).
\end{eqnarray*}
This achieves  the proof of the uniqueness part.
\section{Singular  patches}
\quad In this section, we move on to some results concerning singular vortex patches. Our main goal is to prove   Theorem \ref{the2} and enlarge its statement for  more general initial data belonging to  Yudovich class. To the best of our knowledge, even for the simple case of patches with singular boundary no results on the local well-posedness are known in the literature. In this special case and as it was  previously stressed in  Theorem  \ref{the2} we must take a density with constant magnitude around the singularity. By this assumption we wish to kill the singularity effects  and  reduce  their violent   interaction with    the density which is the main obstacle of this problem. 


\quad The generalization of Theorem \ref{the2} will require  some specific material  that were  developed  by Chemin in \cite{C}. In this new pattern we assume that the initial boundary contains a singular subset and therefore the vector fields which encode the regularity should vanish close to it. This forces us to work with  degenerate vector fields and a cut-off procedure near the singular set becomes necessary.  Therefore we shall deal with infinite family of vector fields parametrized by the distance to the singular set and the control of the blowup with respect to this parameter is mostly the main difficulty in this problem.    
 \subsection{Preliminaries}
 We shall introduce and recall some basic definitions and results in connection with singular vortex patches. These tools are mostly introduced in \cite{C} with sufficient details  and for the completeness of the manuscript we shall recall them here without any proof.
\begin{definition}\label{Defs}
Let  $\Sigma$ be a closed set of the plane.
We denote by $L(\Sigma)$ the set of the functions $v$ such that
$$
\Vert v\Vert_{L(\Sigma)}\triangleq\sup_{0<h\leq e^{-1}}\frac{\Vert v\Vert_{L^\infty(\Sigma_h^c)}}{-\log h}<\infty.
$$
In this definition and for the remaing of the paper we shall adopt the following notation: For $h>0$ $$
 \Sigma_h=\big\{x\in\RR^2;\quad \textnormal{dist}(x,\Sigma)\leq h\big\}\quad\hbox{and} \quad \Sigma_h^c=\big\{x\in\RR^2;\quad \textnormal{dist}(x,\Sigma)\geq h\big\}.
$$

\end{definition}
Next we  introduce log-Lipschitz space which is frequently used in the framework of Yudovich solutions. This space appears in a natural way  thanks  to the fact the velocity associated to a bounded and integrable vorticity  is not in general Lipschitz but belongs to a slight bigger one  called log-Lipschitz class. 
\begin{definition}
We denote  by ${LL}$ the space of log-Lipschitz functions, that is the set of bounded functions $v$ in $\RR^2\to\RR$ satisfying
$$
\Vert v \Vert_{LL}\triangleq \Vert v\Vert_{L^\infty}+\underset {0<\vert x- y\vert<1}{\sup}\frac{\vert v(x)-v(y)\vert}{\vert x-y\vert  \log \frac{e}{\vert x-y \vert} }<+\infty.
$$
\end{definition}
We have the following classical estimate which is a simple consequence of the embedding $B_{\infty,\infty} ^1\subset LL$ combined with Bernstein inequality and Biot-Savart law \eqref{b-s}.
\begin{lemma}\label{lem3}
For any finite $a> 1$ we have
$$
\Vert v\Vert_{LL}\leq C \Vert \omega\Vert_{L^a\cap L^\infty},
$$
with $C$ depending only on $a$.
\end{lemma}

\quad It is well-known, thanks to  Osgood lemma, that a vector field $v$ belonging to the space  ${LL}$ has a unique global  flow map $\psi$ in the class of continuous functions on the   space and time variables. This map is defined by the nonlinear integral equation,
$$
 \psi(t,x)=x+\int_0^tv(\tau,\psi(\tau,x))d\tau \quad\forall(t,x)\in\RR_+\times\RR^2.
 $$
  For more details about this issue  we refer the reader to Section 3.3 in \cite{B-C-D}. 

The next  result deals with some general aspect of  the dynamics of a given set through the  flow associated to a vector field in the $LL$ space. Such result was proved in \cite{C}.

\begin{lemma}\label{s2lem1}
 Let $A_0$ be a subset of $\RR^2$ and $v$ be a vector field belonging  to $L^1_{loc}(\RR_+ ;{LL})$. We denote by $\psi(t)$ the flow associated to this vector field.
 Then setting $A(t) \triangleq\psi(t,A_0)$ we get,
$$
\psi\big(t,(A_0)_h^c\big) \subset \big(A(t)\big)^c_{\delta_t(h)},\quad \textnormal{with}\quad \delta_t(h) \triangleq h^{\exp\int_0^t\Vert v(\tau)\Vert_{LL}d\tau}.
$$
For all $0 \leq \tau \leq t$,
$$
\psi\big(\tau,\psi^{-1}\big(t,(A_t)_h^c)\big) \subset \big(A(\tau)\big)^c_{\delta_{\tau,t}(h)},\quad \textnormal{with}\quad \delta_{\tau,t}(h) \triangleq h^{\exp\int_\tau^t\Vert v(\sigma)\Vert_{LL}d\sigma}.
$$
\end{lemma}

Next, we discuss the regularity  persistence  for a  transport model and  the proof can be found in \cite{C}. 
\begin{proposition}\label{p11}
Let $\varepsilon\in(-1,1)$, $a\in(1,+\infty)$ and $v$ be a smooth divergence-free vector field.  Set
$$
W(t)\triangleq\Big(\Vert\nabla v(t)\Vert_{L(\Sigma_t)}+\Vert \omega(t)\Vert_{L^a \cap L^\infty} \Big)\exp\bigg(\int_0^t\Vert v(\tau)\Vert_{LL}d\tau\bigg), \quad \Sigma_t=\psi(t,\Sigma_0).
$$
Let  $f\in L^\infty_{loc}([0,T],C^\varepsilon)$ be  a solution of transport model,
\begin{equation*}
\left\{ \begin{array}{ll}
\partial_{t}f+v\cdot\nabla f =g, &\\
f_{| t=0}=f_{0},
\end{array} \right.
\end{equation*} 
where $g=g_1+g_2$ is given and belongs to  $L^1([0,T]; C^\varepsilon)$.
 We assume that  $\hbox{supp }f_0\subset (\Sigma_0)_h^c$ and  $\hbox{supp }g(t)\subset (\Sigma_t)_{\delta(t,h)}^c$ for any $t\in[0,T]$, and for some small $h$
$$
\Vert g_2(t)\Vert_\varepsilon\leq -CW(t)\Vert f(t)\Vert_\varepsilon\log h.
$$
Then  the following inequality holds true
$$
\Vert f(t)\Vert_\varepsilon\leq \Vert f_0\Vert_\varepsilon h^{-C\int_0^t W(\tau)d\tau}+\int_0^t h^{-C\int_\tau^t W(\tau')d\tau'}\Vert g_1(\tau)\Vert_\varepsilon d\tau.
$$
Here the constant $C$ is universal and does not depend  on $h$.
 \end{proposition}

 Next, we recall the following definition introduced in \cite{C}.
\begin{definition}\label{def11}
Let $\Sigma$  be a closed subset of $\RR^d$ and $\Xi = (\alpha, \beta, \gamma)$ be a triplet of real numbers. We consider a family $\mathcal{X} = (X_{\lambda,h})_{(\lambda,h)\in \Lambda\times ]0,e^{-1}]}$ of vector fields  belonging to $C^\varepsilon$ as well as their divergences, with $\varepsilon\in]0,1[$ and we denote by $\mathcal{X}_h =(X_{\lambda,h})_{\lambda\in\Lambda}$.

The  familly $\mathcal{X}$ will be said  $\Sigma-$admissible of order $\Xi$ if and only if the following properties are satisfied:
\begin{equation*}
\forall (\lambda, h)\in\Lambda\times ]0,e^{-1}],\ \textnormal{supp}X_{\lambda,h}\subset\Sigma_{h^\alpha}^c,
\end{equation*}
\begin{equation*}
 \inf_{ h\in]0,e^{-1}]}h^\gamma I(\Sigma_h,\mathcal{X}_h)>0,
\end{equation*}
\begin{equation*}
 \sup_{ h\in]0,e^{-1}]}h^{-\beta}N_\varepsilon(\Sigma_h,\mathcal{X}_h)<\infty,
\end{equation*}
where we adopt the following notation: for  $\eta\geq h^\alpha$,
$$
I(\Sigma_\eta,\mathcal{X}_h)\triangleq \inf_{x\in\Sigma_\eta^c}\sup_{\lambda\in\Lambda}\vert X_{\lambda,h}(x)\vert\quad\textnormal{and}\quad N_\varepsilon(\Sigma_\eta,\mathcal{X}_h)\triangleq\sup_{\lambda\in\Lambda}\frac{\tilde{\Vert} X_{\lambda,h}\Vert_{\varepsilon}}{I(\Sigma_\eta,\mathcal{X}_h)}\cdot
$$
\end{definition}
\begin{remark}
Concretely,   the family of vector fields $\mathcal{X}$ that we shall work with  vanishes  near the singular set and therefore we should get   $\gamma, \beta<0$. Moreover the parameter $\alpha>1.$
\end{remark} 
Similarly to the smooth patches we shall introduce for  $\eta\geq h^\alpha$,
\begin{equation}\label{defre}
\Vert u\Vert^{\varepsilon+k}_{\Sigma_\eta,\mathcal{X}_h}\triangleq  N_\varepsilon(\Sigma_\eta,\mathcal{X}_h)\sum_{\vert \alpha\vert\leq  k} \Vert  \partial^\alpha u\Vert_{L^\infty}+\displaystyle{\sup_{\lambda\in\Lambda}}\, \frac{\Vert \pxl u\Vert_{\varepsilon+k-1}}{I(\Sigma_\eta,\mathcal{X}_h)}.
\end{equation}
\subsection{General statement}
We intend now to extend the result of Theorem \ref{the2} and see in turn how to deduce  the result of this theorem. The proof of the general statement will be carried out  in multiple steps and will be postponed in the next subsections. 
\begin{theorem}\label{th2}
Let $0 <\varepsilon <1$, $0<r<e^{-1}$,  $1<a<2$ and $\Sigma_0$ be a closed subset of the plane. Let $v_0$ be a divergence-free vector field with  vorticity  $\omega_ 0$ belonging to $ L^{a}\cap L^\infty$   and  $\rho_0$ be a real-valued function in $W^{1,a}\cap W^{1,\infty}$ and taking constant  value on $(\Sigma_0)_r$.  
 Consider  $\mathcal{X}_0=(X_{0,\lambda,h})_{(\lambda,h)\in \Lambda\times ]0,e^{-1}]}$  a  family of vector fields of class $C^\varepsilon$ as well as their divergences and suppose that this family is  $\Sigma_0$-admissible of order $\Xi_0 = (\alpha_0, \beta_0, \gamma_0)$ such that
\begin{equation*}
\sup_{h\in]0,e^{-1}]}h^{-\beta_0}\Vert \rho_0\Vert^{\varepsilon+1}_{(\Sigma_0)_h,(\mathcal{X}_0)_h}+\sup_{h,\in]0,e^{-1}]}h^{-\beta_0}\Vert \omega_0\Vert^{\varepsilon}_{(\Sigma_0)_h,(\mathcal{X}_0)_h}<\infty.
\end{equation*}
 Then, there exists $T>0$ such that the Boussinesq system \eqref{omega} has a unique solution
 $$(\omega,\rho)\in L^\infty\big([0,T],L^a\cap L^\infty\big)\times L^\infty\big([0,T], W^{1,a}\cap W^{1,\infty}\big).
 $$  
 In addition, we have
$$
\sup_{h\in(0,e^{-1}]}\frac{\Vert \nabla v(t)\Vert_{L^\infty(\Sigma(t)_h^c)}}{-\log h}\in L^\infty([0,T]),
$$
where $\Sigma(t)=\psi(t,\Sigma_0)$.
\end{theorem} 
$\bullet$ {\it Proof of Theorem $\ref{the2}$}. 
Let us briefly show how this result leads to Theorem  \ref{the2} stated in the Introduction.
Let $\Omega_0$ be a bounded open set whose   boundary belongs to $C^{\varepsilon+1}$ outside the closed singular set $\Sigma_0.$ In view of the Definition \ref{bord} we may show the existence of a neighborhood $V_0$ of $\partial\Omega_0$ and a real function $f_0\in C^{\varepsilon+1}$ such that  $\partial\Omega_0=f_0^{-1}(0)\cap V_0$ and whose gradient does not vanish on $V_0\backslash\Sigma_0$.
 We also assume that there exists a positive number  $\tilde{\gamma}_0>0$ such that for all $x\in V_0$,
\begin{equation}\label{H}
  \vert \nabla f_0(x)\vert\geq d(x,\Sigma_0)^{\tilde{\gamma}_0}.
\end{equation}
This means that the curves defining the boundary of $\Omega_0$ are not tangent to one another at infinite order at the singular points.
Consider $(\theta_h)_{h\in(0,e^{-1}]}$  a family of infinitely differentiable functions, supported in $(\Sigma_0)^c_{h/2} $ and taking the value 1 on  the set $(\Sigma_0)^c_h$ and satisfying  for all $h\in]0,e^{-1}]$  and any positive real number $r$,
 $$
 \Vert\theta_h\Vert_r\leq C_rh^{-r}.
 $$
 The existence of such functions can be proved by dilation.
Consider $\tilde{\alpha}$  a  function of class $C^\infty$   supported in $V_0$ and  taking the value $1$ on $V_1$, where $V_1$ is a neighborhood of $\partial\Omega_0$ such that $V_1\subset\subset V_0$. We define the family $\mathcal{X}_0=(X_{0,\lambda,h})_{\lambda\in\{0,1\},h\in]0,e^{-1}]}$ of vector fields by:
$$
X_{0,0,h}=\nabla^\perp(\theta_h f_0),\quad X_{0,1,h}=(1-\tilde{\alpha})\begin{pmatrix}1\\ 0\end{pmatrix}.
$$
It comes to see if the  family $\mathcal{X}_0$ is $\Sigma_0$-admissible of  certain order $\Xi_0 = (\alpha_0, \beta_0, \gamma_0)$.
The first vector field is of class $C^\varepsilon$ with zero divergence and the second is $C^\infty$. On other hand, by construction, $\hbox{supp }X_{0,i,h}\subset (\Sigma_0)_{h/2}^c \subset (\Sigma_0)_{h^\alpha_0}^c$ with $\alpha_0 > 1$. Moreover, thanks to the hypothesis \eqref{H} we may choose $\gamma_0=-\tilde{\gamma}_0$. Finally, we easily  show that
$$
\tilde{\Vert} X_{0,\lambda,h}\Vert_\varepsilon\leq Ch^{-\varepsilon-1}.
$$
Hence, it suffices to take  $\beta_0=\gamma_0-\varepsilon-1$ and $\Xi_0 = (\alpha_0, \beta_0, \gamma_0)$.
Besides, we have
$$
\partial_{X_{0,0,h}}\omega_0=\theta_h\partial_{\nabla^\perp f_0}\omega_0+f_0\partial_{\nabla^\perp \theta_h}\omega_0.
$$
First we observe that  the derivative of $\omega_0$ in  the direction $\nabla^\perp f_0$ is zero and second $\partial_{\nabla^\perp \theta_h}\omega_0$ is a distribution of order zero supported on the boundary $\partial\Omega_0$. As the function  $f_0$ vanishes on $\partial\Omega_0$ then 
$$
f_0\partial_{\nabla^\perp \theta_h}{1}_{\Omega_0}=0
$$ 
Thus we deduce that $\partial_{X_{0,0,h}}\omega_0=0$.  For the second vector field we use that  $1-\tilde{\alpha}$ vanishes on a small neighborhood of $\partial\Omega_0$ and therefore  $\partial_{X_{0,1,h}}\omega_0=0$.
 It remains to check the regularity assumption on the density $\rho_0. $  This function is constant in a neighborhood of $\Sigma_0$ and thus $\, \nabla\rho_0\nabla^\perp\theta_h=0$ and moreover $\partial_{\nabla^\perp f_0}\rho_0\in C^\varepsilon$.  It is then immediate that $\partial_{X_{0,0,h}}\rho_0\in C^\varepsilon$.
Hence the hypothesis of \mbox{Theorem \ref{th2}} are satisfied and the local well-posedness  result  of Theorem \ref{the2} is now  established. To infer that the boundary $\Omega_t$ is a curve of class $C^{1+\varepsilon}$ outside $\Sigma(t)$ we argue as  in the case of regular vortex patches seen in the previous section.
\subsection{A priori estimates} 
This section is devoted to some a priori estimates of $L^p$ type for both the density  and the vorticity  functions. As the velocity may loose regularity and becomes rough close   to  the singular set, our assumption to work with constant density near this  set  seems to be crucial and  unavoidable  in our analysis. Without this  assumption the problem remains open and may be one should expect to propagate the regularity with some loss.
\begin{proposition}\label{peor}
Let  $\Sigma_0$ be a closed set of $\RR^2$ and $(v,\rho)$ be a smooth solution of the  system \eqref{B} 
defined on the time interval $[0,T]$. We suppose that $\rho_0$ is constant in the set $(\Sigma_0)_r=\big\{x\in \RR^2;\, d(x,\Sigma_0)\leq r\}$ for some $r\in (0, e^{-1})$. Then for all $p\in[1,+\infty]$ and for any $t\in [0,T]$ we have
\begin{equation*}
\Vert\nabla \rho(t)\Vert_{L^p}\leq \Vert\nabla \rho_0\Vert_{L^p} r^{-C\int_0^t W(\tau)d\tau}
\end{equation*}
and
\begin{eqnarray*}
\Vert\omega(t)\Vert_{L^p}&\leq &\Vert\omega_0\Vert_{L^p}+t\,\Vert\nabla \rho_0\Vert_{L^p}r^{-C\int_0^t W(\tau)d\tau},
\end{eqnarray*}
with $C$ an absolute constant.
\end{proposition}
\begin{proof}
In order to prove the first estimate, we apply the partial derivative $\partial_j$ to the second equation of the system \eqref{B},
\begin{equation}
\partial_t\partial_j\rho +v\cdot\nabla(\partial_j \rho)=\partial_j v\cdot\nabla\rho.
\end{equation}
Hence, for all $1\leq p\leq \infty$ one has
\begin{equation}
\Vert\partial_j \rho(t)\Vert_{L^p}\leq \Vert\partial_j \rho_0\Vert_{L^p} +\int_0^t\Vert\partial_j v\cdot\nabla\rho(\tau)\Vert_{L^p}d\tau.
\end{equation}
Since $\rho_0$ is transported by the flow $\psi$,
$$
\rho(\tau,x)=\rho_0(\psi^{-1}(t,x)),
$$
then $\rho(\tau)$ is constant in $\psi\big(\tau,(\Sigma_0)_r\big)$ and therefore,
$$
\hbox{supp }\nabla\rho(\tau)\subset \psi\big(\tau,(\Sigma_0)_r\big)^c=\psi\big(\tau,(\Sigma_0)_r^c\big).
$$
Using  Lemma \ref{s2lem1} we get easily 
$$
\hbox{supp }\nabla\rho(\tau)\subset \big(\Sigma_\tau\big)_{\delta_\tau(r)}^c,\quad \delta_\tau(r)\triangleq r^{\exp\big(\int_0^\tau\Vert v(\sigma)\Vert_{LL}d\sigma\big)}.
$$ 
 Accordingly we obtain
\begin{eqnarray*}
\Vert\partial_j v\cdot\nabla\rho(\tau)\Vert_{L^p}\leq \Vert\nabla v(\tau)\Vert_{L^\infty\big((\Sigma_\tau)_{\delta_\tau(r)}^c\big)}\Vert\nabla\rho(\tau)\Vert_{L^p}.
\end{eqnarray*}
Thus coming back to the Definition \ref{Defs}  we may write
\begin{eqnarray*}
\Vert\nabla v(\tau)\Vert_{L^\infty\big((\Sigma_\tau)_{\delta_\tau(r)}^c\big)}&\leq & -\Vert\nabla v(\tau)\Vert_{L(\Sigma_\tau)}\log\delta_\tau(r)\\ &\leq & -(\log r)\Vert\nabla v(\tau)\Vert_{L(\Sigma_\tau)}\exp\Big(\int_0^\tau\Vert v(\sigma)\Vert_{LL}d\sigma\Big) \\ &\leq & -W(\tau)\log r.
\end{eqnarray*}
Recall that the function $W(t)$ was introduced in Proposition \ref{p11}.
It follows that
\begin{equation*}
\Vert\nabla \rho(t)\Vert_{L^p}\leq \Vert\nabla \rho_0\Vert_{L^p} -C\log r\int_0^t\Vert\nabla\rho(\tau)\Vert_{L^p}W(\tau)\ d\tau.
\end{equation*}
By Gronwall lemma we conclude that
\begin{equation*}
\Vert\nabla \rho(t)\Vert_{L^p}\leq \Vert\nabla \rho_0\Vert_{L^p} r^{-C\int_0^t W(\tau)d\tau}.
\end{equation*}
Therefore, in view of the vorticity equation \eqref{omega} we obviously have  for all $1\leq p\leq \infty$,
\begin{eqnarray*}
\Vert\omega(t)\Vert_{L^p}&\leq & \Vert\omega_0\Vert_{L^p}+\int_0^t\Vert \nabla\rho(\tau)\Vert_{L^p} d\tau\\ &\leq &\Vert\omega_0\Vert_{L^p}+\Vert\nabla \rho_0\Vert_{L^p}r^{-C\int_0^t W(\tau)d\tau}t. 
\end{eqnarray*}
This completes the proof of the proposition.
\end{proof}
Now we shall discuss the a priori estimates which are the key of the proof of Theorem \ref{th2}. 
\begin{proposition}\label{prop22}
Let $0<\varepsilon<1$, $a>1$ and $\Sigma_0$ a closed set of the plane and $X_0$ be  a  vector field of class $C^\varepsilon$ as well as its divergence and whose support  is embedded in $(\Sigma_0)_h^c$. Let $(v,\rho)$ be a smooth  solution of the system (\ref{B}) defined  on a  time interval  $[0,T]$ and with the  initial data $(v_0,\rho_0)$ such that $\rho_0$ is constant in the set $(\Sigma_0)_r$. Let $ X_t$ be the  solution of :
\begin{equation}\label{31}
\left\{ \begin{array}{ll}
\big(\partial_t +v\cdot\nabla\big)X_t=\pxt v. &\\
X_{| t=0}=X_{0}.
\end{array} \right.
\end{equation} 
then we have the estimates,
\begin{equation*}
\textnormal{supp}\, X_t\subset \big(\Sigma_t\big)_{\delta_t(h)}^c\quad \textnormal{with}\quad \delta_t(h)\triangleq h^{\exp\big(\int_0^t\Vert v(\tau)\Vert_{LL}d\tau\big)}.
\end{equation*}
\begin{equation*}
\Vert \textnormal{div}X_t\Vert_{\varepsilon}\leq \Vert \textnormal{div}X_{0}\Vert_{C^\varepsilon}h^{-C\int_0^t W(\tau)d\tau},
\end{equation*}
\begin{eqnarray*}
\tilde{\Vert} X_t\Vert_{\varepsilon} +\Vert\pxt \omega(t)\Vert_{\varepsilon-1} &\leq & C e^{Ct}\Big(\tilde{\Vert} X_0\Vert_{ \varepsilon}+\Vert \pxz\omega_0\Vert_{\varepsilon-1}+\Vert \pxz\rho_0\Vert_{\varepsilon}\Big)h^{-C\int_0^t W(\tau)d\tau}\notag\\ &\times &\exp\Big(Ct\Vert \nabla \rho_0\Vert_{L^\infty}r^{-C\int_0^t W(\tau)d\tau}\Big).
\end{eqnarray*}
\begin{equation*}
\Vert \pxt \rho(t)\Vert_{\varepsilon}\leq \Vert \pxt \rho_{0}\Vert_{C^\varepsilon}r^{-C\int_0^t W(\tau)d\tau},
\end{equation*}
where
$$
W(t)\triangleq\Big(\Vert\nabla v(t)\Vert_{L(\Sigma(t))}+\Vert \omega(t)\Vert_{L^a \cap L^\infty} \Big)\exp\bigg(\int_0^t\Vert v(\tau)\Vert_{LL}d\tau\bigg).
$$
\end{proposition}
\begin{proof}
The embedding result on the support  of $X_t$  can be deduced from  Lemma \ref{s2lem1} and the complete proof can be found \mbox{ in \cite{C}. } The second result concerning the estimate of  $\textnormal{div }X_{t}$ follows from the equation 
$$
\big(\partial_t +v\cdot\nabla\big)\textnormal{div}\, X_{t}=0,
$$
and  Proposition \ref{p11}. As to the estimate  of $\Vert X_{t}\Vert_\varepsilon$ we shall  admit the following assertion and  more details see for instance Chapter 9 from \cite{C}.
 $$
  \partial_{X_{t}} v(t)=g_1(t)+g_2(t),
  $$
  with
\begin{eqnarray*}
g_1(t) &\leq & C\Vert \partial_{X_{t}} \omega(t)\Vert_{\varepsilon-1}+C\Vert\textnormal{div}X_{t}\Vert_\varepsilon\Vert\omega(t)\Vert_{L^\infty},
\end{eqnarray*}
and
\begin{eqnarray*}
g_2(t) &\leq & -C\Vert X_{t}\Vert_\varepsilon W(t)\log h.
\end{eqnarray*}
Then applying once  again Proposition \ref{p11} to the equation \eqref{31} we get
\begin{eqnarray*}
\Vert X_{t}\Vert_{C^\varepsilon}\lesssim \Vert X_{0}\Vert_{\varepsilon}h^{-C\int_0^t W(\tau)d\tau}&+&\int_0^t\Vert\textnormal{div}X_{\tau}\Vert_\varepsilon\Vert\omega(\tau)\Vert_{L^\infty} h^{-C\int_\tau^t W(\tau')d\tau'}d\tau\\ & +&\int_0^t\Vert \partial_{X_{\tau}}\omega(\tau)\Vert_{\varepsilon-1}  h^{-C\int_\tau^t W(\tau')d\tau'}d\tau.
\end{eqnarray*}
According to the definition of the function $W$ we may write
$$
\int_0^t\Vert\omega(\tau)\Vert_{L^\infty}d\tau \leq h^{-C\int_0^t W(\tau)d\tau}.
$$
This together with the estimate of $\|\hbox{div} X_t\|_{C^\varepsilon}$  yields
\begin{eqnarray}\label{xt2}
\tilde{\Vert} X_{t}\Vert_{\varepsilon}\lesssim \tilde{\Vert} X_{0}\Vert_{\varepsilon}h^{-C\int_0^t W(\tau)d\tau}+ \int_0^t\Vert \partial_{X_{\tau}}\omega(\tau)\Vert_{\varepsilon-1}  h^{-C\int_\tau^t W(\tau')d\tau'}d\tau.
\end{eqnarray}
Now applying the operator $\pxt$ to the vorticity  equation (\ref{omega}) gives
\begin{equation*}
\big(\partial_t +v\cdot\nabla\big)\partial_{X_{t}}\omega=\partial_{X_{t}}\partial_1\rho,
\end{equation*}
then  from  Proposition \ref{p11} and taking advantage of the inequality \eqref{pp5} we get
\begin{eqnarray*}
\Vert \partial_{X_{t}}\omega(t)\Vert_{ \varepsilon-1}&\lesssim &\Vert \partial_{X_{0}}\omega_0\Vert_{ \varepsilon-1}h^{-C\int_0^t W(\tau)d\tau}+\int_0^t\Vert  \partial_{X_{\tau}}\partial_1\rho(\tau)\Vert_{\varepsilon-1}h^{-C\int_\tau^t W(\tau')d\tau'}d\tau\notag\\ &\lesssim & \Vert \partial_{X_{0}}\omega_0\Vert_{ \varepsilon-1}h^{-C\int_0^t W(\tau)d\tau}+\int_0^t\Vert  \partial_{X_{\tau}}\rho(\tau)\Vert_{\varepsilon}h^{-C\int_\tau^t W(\tau')d\tau'}d\tau\\
&+&\int_0^t\Vert\nabla\rho(\tau)\Vert_{L^\infty}\tilde{\Vert} X_{\tau}\Vert_{\varepsilon}h^{-C\int_\tau^t W(\tau')d\tau'}d\tau.
\end{eqnarray*}
As it has been  shown for  smooth patches 
 the scalar function   $\partial_{X_{t}}\rho(t)$ is transported by the flow, 
 $$
 (\partial_t+v\cdot\nabla )\partial_{X_{t}}\rho(t)=0.
 $$ 
 It is easy to check that  $\partial_{X_{0}}\rho_0$ is supported  in $(\Sigma_0)_{r}^c$ and thus  Proposition \ref{p11} gives 
$$
\Vert  \partial_{X_{\tau}}\rho(\tau)\Vert_{\varepsilon}\leq \Vert  \partial_{X_0}\rho_0\Vert_{\varepsilon}\,r^{-C\int_0^\tau W(\tau')d\tau'}.
$$
 Hence, we obtain 
\begin{eqnarray*}
\Vert \partial_{X_{t}}\omega(t)\Vert_{ \varepsilon-1}&\lesssim& \big(\Vert \partial_{X_{0}}\omega_0\Vert_{ \varepsilon-1}+\Vert  \partial_{X_{0}}\rho_0\Vert_{ \varepsilon}t\big)h^{-C\int_0^t W(\tau)d\tau}\\
&+&\int_0^t\Vert\nabla\rho(\tau) \Vert_{L^\infty}\tilde{\Vert} X_{\tau}\Vert_{\varepsilon}h^{-C\int_\tau^t W(\tau')d\tau'}d\tau.
\end{eqnarray*}
Putting together the preceding estimate and  \eqref{xt2}   we find
\begin{equation*}
\Gamma(t)\lesssim  \Gamma(0)+\Vert \partial_{X_{0}}\rho_0\Vert_{ \varepsilon}t+\int_0^t\big(\Vert \nabla\rho(\tau)\Vert_{L^\infty}+1)\Gamma(\tau)d\tau,
\end{equation*}
where  $$\Gamma(t)\triangleq\big(\Vert \partial_{X_{t}}\omega(t)\Vert_{\varepsilon-1}+\tilde{\Vert} X_{t}\Vert_{ \varepsilon}\big)h^{C\int_0^t W(\tau)d\tau}.$$ 
So Gronwall lemma ensures that
\begin{equation}\label{gam}
\Gamma(t)\leq \big(\Gamma(0)+\Vert  \partial_{X_{0}}\rho_0\Vert_{ \varepsilon}\big)e^{C\int_0^t\Vert \nabla \rho(\tau)\Vert_{L^\infty}d\tau}e^{Ct}.
\end{equation}
Then Proposition \ref{peor} completes the proof.
\end{proof}
\quad Now, we have to control the Lipschitz norm  of the velocity  outside the transported of the singular set $\Sigma_0$ by the flow.
\begin{proposition}\label{prop220}
Let $0<\varepsilon<1$,  $0<r<e^{-1}$ and $a>1$ and $\Sigma_0$ be a closed set of the plane. Let $\mathcal{X}_0=\big(X_{0,\lambda,h}\big)_{\lambda\in\Lambda,h\in(0,e^{-1}]}$ be a family  vector field  which is   $\Sigma_0-$admissible of order $\Xi_0=(\alpha_,\beta_0,\gamma_0)$ and  let  $(v,\rho)$ be a smooth solution of the  \mbox{system $(\ref{B})$} defined  on a  time interval  $[0,T^\star[$. We assume that  the initial data satisfy $\omega_0,\,\nabla\rho_0\in L^a $, $\rho_0$ is constant in the set $(\Sigma_0)_r$ and 
 \begin{equation*}
\sup_{h\in]0,e^{-1}]}h^{-\beta_0}\Vert \rho_0\Vert^{\varepsilon+1}_{(\Sigma_0)_h,(\mathcal{X}_0)_h}+\sup_{h,\in]0,e^{-1}]}h^{-\beta_0}\Vert \omega_0\Vert^{\varepsilon}_{(\Sigma_0)_h,(\mathcal{X}_0)_h}<\infty.
\end{equation*}
Then there exists $0<T<T^\star $ such that 
$$
\Vert \nabla v(t)\Vert_{L(\Sigma(t))} \in L^\infty\big([0,T]\big).
$$
\end{proposition}
\begin{proof}
The dynamical vector fields $ \{X_{t,\lambda,h}\}$ are nothing but the transported of the initial family by the flow. They are given by the identity 
$$
Y_{t,\lambda,h}(x)\triangleq X_{t,\lambda,h}\big(\psi(t,x)\big)=\partial_{X_{0,\lambda,h}} \psi(t,x)
$$
and clearly they satisfy
$$
\partial_t Y_{t,\lambda,h}(x)=\{\nabla v(t,\psi(t,x))\}\cdot Y_{t,\lambda,h}(x).
$$
Fix $t>0$ and set for $\tau\in [0,t],$ \, $Z(\tau,x)=Y_{t-\tau,\lambda,h}(x)$, then 
$$
\partial_\tau Z(\tau,x)=-\{\nabla v(t-\tau,\psi(t-\tau,x))\}\cdot Z(\tau,x).
$$
Hence using  Gronwall lemma we get
\begin{eqnarray*}
|Z(t,x)|&\le& |Z(0,x)|\, e^{\int_0^t|\nabla v(t-\tau,\psi(t-\tau,x))|d\tau}\\
&\le&|Z(0,x)|\, e^{\int_0^t|\nabla v(\tau,\psi(\tau,x))|d\tau},
\end{eqnarray*}
which is equivalent to
$$
|Y_{0,\lambda,h}(x)|\le |Y_{t,\lambda,h}(x)|e^{\int_0^t|\nabla v(\tau,\psi(\tau,x))|d\tau}.
$$
This gives in turn, 
\begin{equation*}
\vert X_{0,\lambda,h}(\psi^{-1}(t,x))\vert \leq \vert X_{t,\lambda,h}(x)\vert\, e^{\int_0^t\vert \nabla v(\tau,\psi(\tau,\psi^{-1}(t,x)))\vert d\tau}.
\end{equation*}
Denoting by  $\delta_t^{-1}$ the inverse function of $\delta_t$ given by the formula,
$$ \delta_t^{-1}(h)\triangleq h^{\exp(-\int_0^t\Vert v(\tau)\Vert_{LL}d\tau)},$$
 the last estimate  yields
\begin{eqnarray}\label{ii}
\inf_{x\in(\Sigma_t)_{\delta_t^{-1}(h)}^{c}}\sup_{\lambda\in\Lambda}\vert X_{0,\lambda,h}(\psi^{-1}(t,x))\vert &\leq &\inf_{x\in(\Sigma_t)_{\delta_t^{-1}(h)}^{c}}\sup_{\lambda\in\Lambda}\vert X_{t,\lambda,h}(x)\vert\\ &\times &\exp\Big(\int_0^t\big\Vert \nabla v\big(\tau,\psi(\tau,\psi^{-1}(t,\cdot))\big)\big\Vert_{L^\infty((\Sigma_t)_{\delta_t^{-1}(h)}^c)} d\tau\Big)\notag,
\end{eqnarray}
where we recall that  $\Sigma_t=\psi(t, \Sigma_0)$ and $(\Sigma_t)^c_{\eta}=\big\{x\in \RR^2;\, d\big(x,\Sigma_t\big)\geq \eta\big\}$.\\
According to Lemma  \ref{s2lem1} we have 
$$\psi^{-1}\big(t,\big(\Sigma_t\big)_{\delta_t^{-1}(h)}^c\big)\subset \big(\Sigma_0)_{\delta_t(\delta_t^{-1}(h))}^c=(\Sigma_0)_{h}^c.
$$
 Then  we immediately deduce that
\begin{eqnarray}\label{lipsi}
\inf_{x\in(\Sigma_t)_{\delta_t^{-1}(h)}^c}\sup_{\lambda\in\Lambda}\vert X_{0,\lambda,h}(\psi^{-1}(t,x))\vert &= &\inf_{y\in\psi^{-1}(t,(\Sigma_t)_{\delta_t^{-1}(h)}^c)}\sup_{\lambda\in\Lambda}\vert X_{0,\lambda,h}(y)\vert\notag\\ 
&\geq& \inf_{y\in(\Sigma_0)_{h}^c}\sup_{\lambda\in\Lambda}\vert X_{0,\lambda,h}(y)\vert\notag\\ 
&\geq & I\big((\Sigma_0)_{h},(\mathcal{X}_0)_{h}\big).
 \end{eqnarray}
Moreover, in view of the same lemma  we have
$$
\psi\Big(\tau,\psi^{-1}\big(t,(\Sigma_t)_{\delta_t^{-1}(h)}^c\big)\Big)\subset\big(\Sigma_\tau )_{\delta_{\tau,t}(\delta_t^{-1}(h))}^c=\big(\Sigma_\tau)_{\delta_\tau^{-1}(h)}^c\subset\big(\Sigma_\tau)_{h}^c.
$$
Consequently,  we may  write
\begin{eqnarray*}
\big\Vert \nabla v\big(\tau,\psi(\tau,\psi^{-1}(t,\cdot))\big)\big\Vert_{L^\infty\big(({\Sigma_t})_{\delta_t^{-1}(h)}^c\big)}&\leq & \Vert \nabla v(\tau)\Vert_{L^\infty((\Sigma_\tau)_{h}^c)}\notag\\ &\leq& -(\log h)\Vert \nabla v(\tau)\Vert_{L(\Sigma_\tau)}\notag\\ &\leq&  -(\log h)W(\tau).
\end{eqnarray*}
Combining  the last estimate with   \eqref{ii} and\eqref{lipsi} we get
\begin{equation}\label{ixtd}
I\big((\Sigma_t)_{\delta_t^{-1}(h)},(\mathcal{X}(t))_h\big)
 \geq  I\big((\Sigma_0)_{h},(\mathcal{X}_0)_{h}\big)h^{\int_0^t W(\tau)d\tau}
\end{equation}
We introduce
$$
\Upsilon(t)\triangleq  \Vert\omega(t)\Vert_{L^\infty}\tilde{\Vert} X_{t,\lambda,h}\Vert_{\varepsilon}+\Vert \partial_{X_{t,\lambda,h}} \omega(t)\Vert_{\varepsilon-1}.
$$
Then combining  Proposition \ref{prop22}  and  Proposition \ref{peor} we get
\begin{eqnarray*}
\Upsilon(t)&\lesssim & e^{Ct}\big(1+\Vert\omega_0\Vert_{L^\infty}\big)\big(\tilde{\Vert} X_{0,\lambda,h}\Vert_{\varepsilon}+\Vert \partial_{X_{0,\lambda,h}} \omega_0\Vert_{\varepsilon-1}+\Vert \partial_{X_{0,\lambda,h}} \rho_0\Vert_{\varepsilon-1}\big)\\ &\times &\exp\big(Ct\Vert\nabla \rho_0\Vert_{ L^\infty}r^{-C\int_0^t W(\tau)d\tau}\big)h^{-C\alpha_0(\int_0^t W(\tau)d\tau)}.
\end{eqnarray*}
Hence in view of the  Definition \ref{def11} and \eqref{ixtd} we immediately deduce that
\begin{eqnarray}\label{eqom}
\Upsilon(t) &\leq & Ce^{Ct}\Big(N_\varepsilon\big((\Sigma_0)_{h},(\mathcal{X}_0)_{h}\big)+\big(1+\Vert\omega_0\Vert_{L^\infty} \big)\big(\Vert\omega_0\Vert^{\varepsilon}_{(\Sigma_0)_{h},(\mathcal{X}_0)_{h}}+\Vert\rho_0\Vert^{\varepsilon+1}_{(\Sigma_0)_{h},(\mathcal{X}_0)_{h}}\big)\Big)\notag\\ &\times & I\big((\Sigma_0)_{h},(\mathcal{X}_0)_{h}\big) h^{-C\alpha_0\int_0^t W(\tau)d\tau}e^{Ct\Vert\nabla \rho_0\Vert_{ L^\infty}r^{-C\int_0^t W(\tau)d\tau}}\notag\\ 
&\leq &C e^{Ct}\sup_{0<h\leq e^{-1}}h^{-\beta_0}\Big(N_\varepsilon((\Sigma_0)_{h},(\mathcal{X}_0)_{h})+\big(1+\Vert\omega_0\Vert_{L^\infty} \big)\big(\Vert\omega_0\Vert^{\varepsilon}_{(\Sigma_0)_{h},(\mathcal{X}_0)_{h}}+\Vert\rho_0\Vert^{\varepsilon+1}_{(\Sigma_0)_{h},(\mathcal{X}_0)_{h}}\big)\Big)\notag\\ 
&\times & I\big((\Sigma_t)_{\delta_t^{-1}(h)},(\mathcal{X}(t))_h\big)h^{\beta_0-C\alpha_0\int_0^t W(\tau)d\tau} e^{\{Ct\Vert\nabla \rho_0\Vert_{ L^\infty}r^{-C\int_0^t W(\tau)d\tau}\}}\notag.
\end{eqnarray}

From this estimate  and   the definition (\ref{defre}), one has
\begin{eqnarray}\label{a}
\Vert\omega(t)\Vert_{(\Sigma(t))_{\delta_t^{-1}(h)},(\mathcal{X}(t))_{h}}^{\varepsilon}&\leq &C_0e^{Ct} h^{\beta_0-C\int_0^t W(\tau)d\tau} e^{\{Ct\Vert\nabla \rho_0\Vert_{ L^\infty}r^{-C\int_0^t W(\tau)d\tau})\}}.
\end{eqnarray}
Now, we shall  combine  Theorem \ref{propoo1}   with the Proposition \ref{peor}
and    the monotonicity of the \mbox{map $x\longmapsto x\log \big(e + \frac{a}{x}\big)$} to get
\begin{eqnarray*}
\Vert\nabla v(t)\Vert_{L^\infty((\Sigma_t)_{\delta_t^{-1}(h)}^c)}\leq  C\Big(\Vert\omega_0\Vert_{L^a\cap L^\infty}+t\Vert\nabla \rho_0\Vert_{L^a\cap L^\infty}r^{-C\int_0^t W(\tau)d\tau}\Big)\log\bigg(e+\frac{ \Vert\omega\Vert_{(\Sigma_t)_{\delta_t^{-1}(h)},(\mathcal{X}(t))_{h}}^{\varepsilon}}{\Vert \omega_0\Vert_{L^\infty}}\bigg).
\end{eqnarray*}
So according to the estimate (\ref{a}), we find
\begin{eqnarray*}
\Vert\nabla v(t)\Vert_{L^\infty((\Sigma_t)_{\delta_t^{-1}(h)}^c)}&\leq & C\Big(\Vert\omega_0\Vert_{L^1\cap L^\infty}+t\Vert\nabla \rho_0\Vert_{L^a\cap L^\infty}r^{-C\int_0^t W(\tau)d\tau}\Big)\\ &\times &\bigg(C_0+t+t\Vert\nabla \rho_0\Vert_{L^a\cap L^\infty}r^{-C\int_0^t W(\tau)d\tau}+\Big(\beta_0 -C\int_0^t W(\tau)d\tau\Big)\log h\bigg).
\end{eqnarray*}
It follows  that
\begin{eqnarray*}
\frac{\Vert\nabla v(t)\Vert_{L^\infty((\Sigma_t)_{\delta_t^{-1}(h)}^c)}}{-\log \delta_t^{-1}(h)}&\leq & C\Big(\Vert\omega_0\Vert_{L^1\cap L^\infty}+t\Vert\nabla \rho_0\Vert_{L^a\cap L^\infty}r^{-C\int_0^t W(\tau)d\tau}\Big)\\ &\times &\Big(C_0+t+t\Vert\nabla \rho_0\Vert_{ L^\infty}r^{-C\int_0^t W(\tau)d\tau} +\int_0^t W(\tau)d\tau\Big)e^{\int_0^t\Vert v(\tau)\Vert_{LL}d\tau}.
\end{eqnarray*}
By the definition of $W(t)$ introduced in the Proposition \ref{p11} we have
\begin{eqnarray*}
W(t)&\leq & C\Big(\Vert\omega_0\Vert_{L^a\cap L^\infty}+\Vert\nabla \rho_0\Vert_{L^a\cap L^\infty}r^{-C\int_0^t W(\tau)d\tau}t\Big)\\ &\times &\Big(C_0+t +\Vert\nabla \rho_0\Vert_{ L^\infty}r^{-C\int_0^t W(\tau)d\tau}t +\int_0^t W(\tau)d\tau\Big)e^{2\int_0^t\Vert v(\tau)\Vert_{LL}d\tau}.
\end{eqnarray*}
We choose $T$ such that
\begin{equation}\label{t1}
T\Vert\nabla \rho_0\Vert_{L^a\cap L^\infty}r^{-C\int_0^T W(\tau)d\tau}\leq  \textnormal{min}\big(1,\Vert\omega_0\Vert_{L^1\cap L^\infty}\big).
\end{equation}
From Lemma \ref{lem3} and Proposition \ref{peor} we get for $t\in [0,T]$
\begin{eqnarray}\label{esll}
\Vert v(t)\Vert_{LL}&\leq &\Vert \omega(t)\Vert_{L^a\cap L^\infty}\notag\\ &\leq &\Vert \omega_0\Vert_{L^a\cap L^\infty}+\Vert\nabla \rho_0\Vert_{L^a\cap L^\infty}r^{-C\int_0^t W(\tau)d\tau}t\notag\\ &\leq &2\Vert \omega_0\Vert_{L^a\cap L^\infty}.
\end{eqnarray}
Then, 
\begin{eqnarray*}
W(t)&\leq & C\Vert\omega_0\Vert_{L^1\cap L^\infty}\Big(C_0+t +\int_0^t W(\tau)d\tau\Big)e^{C\Vert \omega_0\Vert_{L^a\cap L^\infty}t}.
\end{eqnarray*}
Therefore by using  to Gronwall lemma we conclude that for $t\in [0,T]$
\begin{eqnarray*}
W(t) &\leq & C\Vert \omega_0\Vert_{L^a\cap L^\infty}\big(C_0+t\big)e^{C\Vert \omega_0\Vert_{L^a\cap L^\infty}t}\exp\Big(e^{C\Vert \omega_0\Vert_{L^1\cap L^\infty}t}\Big).
\end{eqnarray*}
It follows that
$$
\int_0^t W(\tau)d\tau\le (C_0+t)\exp\Big(e^{C\Vert \omega_0\Vert_{L^a\cap L^\infty}t}\Big),\quad  r^{-\int_0^tW(\tau)d\tau} \leq r^{-(C_0+t)\exp\big(e^{Ct\Vert \omega_0\Vert_{L^a\cap L^\infty}}\big)}.
$$
Hence in order to ensure the assumption \eqref{t1} it suffices to impose,
$$
T\|\nabla\rho_0\|_{L^\infty}r^{-(C_0+T)\exp\big(e^{CT\Vert \omega_0\Vert_{L^a\cap L^\infty}}\big)}=\textnormal{min}\big(1,\Vert\omega_0\Vert_{L^1\cap L^\infty}\big).
$$
The existence of such  $T>0$ can be justified by a continuity argument and this ends  the proof of the proposition.
\end{proof}
\subsection{Existence}
The existence part can be done in a similar way to the case of smooth patches by smoothing out the initial data. However we should be careful about this procedure which must preserve the imposed geometric structure. Especially we have seen in the a priori estimates that a constant density close to the singular set is a  crucial fact and thereby  this must be satisfied for the smooth approximation of the density. Hence  
  we smooth  the   initial velocity as before  by setting $v_{0,n}=S_nv_0$ but for the density  we have to choose a compactly supported mollifiers.  More precisely, we take $\phi\in C_0^\infty(\RR^2)$ a positive function  supported in the ball of center $0$ and radius $1$ and of \mbox{integral $1$} over $\RR^2$. We denote by $(\phi_n)_{n\in\NN}$ the usual mollifiers:
$$
\phi_n(x)=n^2\phi(nx)
$$ 
and we set
$$
 \rho_{0,n}=\phi_n*\rho_0.
$$
Then   the following uniform bounds hold true,
$$
\Vert \rho_{0,n}\Vert_{L^2}\leq\Vert \nabla\rho_0\Vert_{L^2},\quad\Vert \nabla\rho_{0,n}\Vert_{L^1\cap L^\infty}\leq \Vert \nabla\rho_0\Vert_{L^1\cap L^\infty}
$$
and
$$
\Vert \pxzl\rho_{0,n}\Vert_{\varepsilon}\leq\Vert \pxzl\rho_0\Vert_{\varepsilon}+\tilde{\Vert}X_{0,\lambda}\Vert_{\varepsilon}\Vert \nabla\rho_0\Vert_{L^\infty}. 
$$
The  first two estimates are easy to get by using the classical properties of the convolution laws. The proof of the last estimate is obtained by writing  the identity
$$
 \pxzl\rho_{0,n}=\phi_n* (\pxzl\rho_{0})+ [\partial_{X_{0,\lambda}},\phi_n*]\rho_{0},
$$
where we use the notation $[A, B*]f=A(B*f)-B*(Af).$ We can easily check that the  first term is uniformly bounded in $C^\varepsilon$, as to the second one we use 
Proposition \ref{comu} which gives
$$
\|[\partial_{X_{0,\lambda}},\phi_n*]\rho_{0}\|_{C^\varepsilon}\le C\|X_{0,\lambda}\|_{C^\varepsilon}\|\nabla \rho_0\|_{L^\infty}.
$$
 It remains to show that $\rho_{0,n}$ is constant in small neighborhood of $\Sigma_0$. For this aim we consider the set
$$
(\Sigma_0)_{r-\frac{1}{n}}\triangleq \Big\{x\in \RR^2; \,  d(x,\Sigma_0)\le  r-\frac{1}{n}\Big\}.
$$
We can easily check according to the support property of the convolution that $\rho_{0,n}$ is constant in the set $\Sigma_{r-\frac{1}{n}}^0$ which contains  $(\Sigma_0)_{r/2}$ for $n$ big enough. The remaining of the proof is similar to the one of the Theorem \ref{the} and we omit here the details.
\subsection{Uniqueness}
It seems that the uniqueness argument performed in the smooth patches cannot be easily  extended to singular patches because the velocity is not Lipschitz. To avoid this difficulty we shall use the original argument of Yudovich \cite{Y}. First let us observe according to \cite{C} that   the velocity does not necessary belong to $L^2$ but to an affine space of type $\sigma+L^2$. The vector field  $\sigma $ is a stationary solution for Euler equations and can be constructed as follows: let  $g$ be a radial function in $C^\infty_0$ supported away from the origin and set, \begin{equation}\label{sigma}
\sigma(x)=\frac{x^\perp}{\vert x\vert^2}\int_0^{\vert x\vert}rg(r)dr.
\end{equation}
Such $\sigma$ is a smooth stationary solutions of the incompressible Euler system,
$$
\partial_t\sigma=\mathbb{P}(\sigma\cdot\nabla \sigma)=0
$$
where $\mathbb{P}\triangleq\Delta^{-1}\textnormal{div}$ is Leray's projector onto divergence-free vector fields. 
It behaves like $1/\vert x\vert$ at infinity and  $\nabla \sigma$ belongs to $H^s(\RR^2)$ for all $r \in \RR$. For a vortex patch we can show that its velocity given by Biot-Savart law belongs to some $\sigma+L^2$.\\
Let us state the following lemma
\begin{lemma}
Let $\sigma$ be a stationary vector field satisfying \eqref{sigma} and $(v_0,\rho_0)$ be a smooth initial data belonging to $(\sigma+L^2)\times L^2$. Then any local solution  $(v(t),\rho(t))$ of the system \eqref{B} associated to the initial data $(v_0,\rho_0)$ belongs to $(\sigma+L^2)\times L^2$
\end{lemma}
\begin{proof}
Setting $v=u+\sigma$ then  the system\eqref{B} can be written in the form,
\begin{equation}\label{f1}
\left\{ \begin{array}{ll}
\partial_{t}u+(u+\sigma)\cdot\nabla u=-u\cdot\nabla \sigma-\nabla p +\rho\vec{e}_2, &\\
\partial_{t}\rho+(u+\sigma)\cdot\nabla\rho =0, &\\
u_{\vert_{t=0}}=v_0+\sigma,\quad \rho_{\vert_{t=0}}=\rho_0.
\end{array} \right. 
\end{equation} 
Since $\textnormal{div}u=\textnormal{div}\sigma=0$, then we have the following $L^2$ estimates
$$
\Vert u(t)\Vert_{L^2}\leq \Vert u_0\Vert_{L^2}+\Vert \rho_0\Vert_{L^2}t+\Vert \nabla\sigma\Vert_{L^\infty}\int_0^t \Vert u(\tau)\Vert_{L^2}d\tau.
$$
By Gronwall inequality we conclude that
$$
\Vert u(t)\Vert_{L^2}\leq \big(\Vert u_0\Vert_{L^2}+\Vert \rho_0\Vert_{L^2}t\big)e^{\Vert \nabla\sigma\Vert_{L^\infty}t}.
$$
This conclude the proof of the lemma.
\end{proof}
Now we shall prove the uniqueness part. As we have already seen the velocity belongs to $\sigma+L^2$  and the uniqueness in the space $L^\infty_T L^2$ should be done by using the formulation (\ref{f1}). However for the clarity of the proof  we shall  assume that $\sigma=0$ and the  proof works for non trivial  $\sigma$ as well.    Let $(v_1,\rho_1, p_1)$ and $(v_2,\rho_2, p_2)$ two solutions of the system  (\ref{B}) with the same initial data. We notice that $(v,\rho,p)\triangleq (v_2-v_1,\rho_2-\rho_1, p_2-p_1)$ satisfies
\begin{equation}
\left\{ \begin{array}{ll}
\partial_{t}v+v_2\cdot\nabla v=-v\cdot\nabla v_1-\nabla p +\rho\vec{e}_2, &\\
\partial_{t}\rho+v_2\cdot\nabla\rho =-v\cdot\nabla \rho_1, &\\
v_{\vert_{t=0}}=v_0,\quad \rho_{\vert_{t=0}}=\rho_0.
\end{array} \right. 
\end{equation} 
A standard energy method with H\"older inequality yield for all $q\in[a,+\infty[$,
\begin{eqnarray*}
\frac{1}{2}\frac{d}{dt}\Vert v\Vert_{L^2}^2&\leq &\Vert\nabla v_1(t)\Vert_{L^{q}}\Vert v(t)\Vert_{L^{2q'}}^2+\Vert \rho(t)\Vert_{L^2}\Vert v(t)\Vert_{L^2}\\\ &\lesssim & q\Vert\omega_1(t)\Vert_{L^{a}\cap L^\infty}\Vert v(t)\Vert_{L^{\infty}}^{\frac{2}{q}}\Vert v(t)\Vert_{L^{2}}^{\frac{2}{q'}}+\Vert \rho(t)\Vert_{L^2}\Vert v\Vert_{L^2} .
\end{eqnarray*}
with $q'=\frac{q}{q-1}$ and
\begin{eqnarray*}
\frac{1}{2}\frac{d}{dt}\Vert \rho(t)\Vert_{L^2}^2&\le &\Vert\nabla \rho_1(t)\Vert_{L^{\infty}}\Vert v(t)\Vert_{L^{2}}\Vert \rho(t)\Vert_{L^2}.
\end{eqnarray*}
Let $\eta$ be a small parameter and set, 
$$
\Gamma_\eta(t)\triangleq \sqrt{\Vert \rho(t)\Vert_{L^2}^2+\Vert v(t)\Vert_{L^2}^2+\eta}.
$$
Then,
$$
\frac{d}{dt}\Gamma_\eta(t)\le C q\Vert\omega_1(t)\Vert_{L^{a}\cap L^\infty}\Vert v(t)\Vert_{L^{\infty}}^{\frac{2}{q}}\Gamma_\eta(t)^{1-\frac{2}{q}}+(1+\Vert\nabla \rho_1(t)\Vert_{L^{\infty}})\Gamma_\eta(t).
$$
Setting
$$
\Upsilon_\eta(t)\triangleq e^{-\int_0^t(1+\Vert\nabla \rho_1(\tau)\Vert_{L^{\infty}})d\tau}\Gamma_\eta(t),
$$
we obtain
$$
\frac{d}{dt}\Upsilon_\eta(t)\le C  q\Vert\omega_1(t)\Vert_{L^{a}\cap L^\infty}\Vert v(t)\Vert_{L^{\infty}}^{\frac{2}{q}}\Upsilon_\eta(t)^{1-\frac{2}{q}} e^{-\frac{2}{q}\int_0^t(1+\Vert\nabla \rho_1(\tau)\Vert_{L^{\infty}})d\tau}
$$
which gives 
$$
\frac{2}{q}\Upsilon_\eta(t)^{\frac{2}{q}-1}\frac{d}{dt}\Upsilon_\eta(t)\leq C\Vert\omega_1(t)\Vert_{L^{a}\cap L^\infty}\Vert v(t)\Vert_{L^{\infty}}^{\frac{2}{q}} .$$
Integrating in time we get
$$
\Upsilon_\eta(t)\leq\Big(\eta^{\frac{1}{q}}+C\int_0^t\Vert\omega_1(\tau)\Vert_{L^{a}\cap L^\infty}\Vert v(\tau)\Vert_{L^{\infty}}^{\frac{2}{q}} d\tau\Big)^{\frac{q}{2}}.
$$
Letting $\eta$  go to $0$ leads
\begin{eqnarray*}
\Vert \rho(t)\Vert_{L^2}^2+\Vert v(t)\Vert_{L^2}^2&\leq & \Vert v(t)\Vert_{L^\infty_{T_0}L^{\infty}}^{2} \Big(C\int_0^t\Vert\omega_1(\tau)\Vert_{L^{a}\cap L^\infty}d\tau\Big)^{q}.
\end{eqnarray*}
Then, from the Biot-Savart law we have for $a\in (1,2)$
\begin{eqnarray*}
\Vert \rho(t)\Vert_{L^2}^2+\Vert v(t)\Vert_{L^2}^2&\leq & \Vert \omega(t)\Vert_{L^\infty_{T_0}(L^{a}\cap L^\infty}^{2} \Big(C\int_0^t\Vert\omega_1(\tau)\Vert_{L^{a}\cap L^\infty}d\tau\Big)^{q}\\ &\leq & C_0 \Big(C\int_0^t\Vert\omega_1(\tau)\Vert_{L^{a}\cap L^\infty}d\tau\Big)^{q}.
\end{eqnarray*}
Therefore, We may find $T$ such that $\int_0^T\Vert\omega_1(\tau)\Vert_{L^{a}\cap L^\infty}d\tau<\frac{1}{C}$. Letting first  $q$ tend to $+\infty$ and using bootstrap  arguments we can conclude that $(v,\rho)\equiv 0$ on $[0,T_0]$
\section{Appendix}
\begin{proposition}\label{comu}
Given $0<\varepsilon<1$, $X$  a vector field belonging to $C^\varepsilon$ and $f\in Lip(\RR^2)$. Let $\phi\in C_0^\infty(\RR^2)$ be a positive function  supported in the ball of center $0$ and radius $1$ and such that $\int_{\RR^2}\phi(x)dx=1$. For all $n\in\NN$ we  set
$$
\phi_n(x)=n^2\phi(nx),
$$
and we denote by $R_n$ the convolution operator with the function $\phi_n$. Then we have the following estimate
$$
\Vert[\partial_X,R_n]f\Vert_{C^\varepsilon}\leq C\Vert X\Vert_{\varepsilon}\Vert\nabla f\Vert_{L^\infty}.
$$
\end{proposition}
\begin{proof}
By definition we write
$$
[\partial_X,R_n]f(x)=\int_{\RR^2}\phi_n(x-y)\big(X(x)-X(y)\big)\nabla f(y)dy.
$$
Then for all $x_1,x_2\in \RR^2$ such that $\vert x_1-x_2\vert<1$ one has
$$
[\partial_X,R_n]f(x_1)-[\partial_X,R_n]f(x_2)=\hbox{I}+\hbox{II},
$$
where
$$
\hbox{I}=\int_{\RR^2}\phi_n(x_1-y)\big(X(x_1)-X(x_2)\big)\nabla f(y)dy
$$
$$
\hbox{II}=\int_{\RR^2}\big(\phi_n(x_1-y)-\phi_n(x_2-y)\big)\big(X(x_2)-X(y)\big)\nabla f(y)dy.
$$
By straightforward computations  we get,
 $$
\vert\hbox{I}\vert\le\Vert\nabla f\Vert_{L^\infty}\Vert\phi\Vert_{L^1}\vert x_1-x_2\vert^\varepsilon\|X\|_{C^\varepsilon}.
$$
As $\phi$ is supported in the ball of center $0$ and radius $\frac{1}{n}$ then we may write
\begin{eqnarray*}
\vert\hbox{II}\vert &\leq&\Vert\nabla f\Vert_{L^\infty} \int_{B(x_1,\frac{1}{n})\cup B(x_2,\frac{1}{n})}\vert\phi_n(x_1-y)-\phi_n(x_2-y)\vert\vert X(x_2)-X(y)\vert dy\\ &\leq & \Vert\nabla f\Vert_{L^\infty}\big(\hbox{II}_1+\hbox{II}_2+\hbox{II}_3\big) .
\end{eqnarray*}
with
\begin{eqnarray*}
\hbox{II}_1= \int_{B(x_1,\frac{1}{n})}\vert\phi_n(x_1-y)-\phi_n(x_2-y)\vert\vert X(x_2)-X(x_1)\vert dy
 \end{eqnarray*}
 \begin{eqnarray*}
\hbox{II}_2= \int_{B(x_1,\frac{1}{n})}\vert\phi_n(x_1-y)-\phi_n(x_2-y)\vert\vert X(x_1)-X(y)\vert dy
 \end{eqnarray*} 
 and
 \begin{eqnarray*}
\hbox{II}_3= \int_{B(x_2,\frac{1}{n})}\vert\phi_n(x_1-y)-\phi_n(x_2-y)\vert\vert X(x_2)-X(y)\vert dy.
 \end{eqnarray*}
 The term $\hbox{II}_1$ can be treated by the same way as $\hbox{I}$. For term $\hbox{II}_2$   we have
  \begin{eqnarray*}
\hbox{II}_2&\leq & n^{2+\varepsilon}\Vert \phi\Vert_{\varepsilon}\Vert X\Vert_{\varepsilon} \int_{B(x_1,\frac{1}{n})}\vert x_1-y\vert^{\varepsilon} dy\\ &\leq & \Vert \phi\Vert_{\varepsilon}\Vert X\Vert_{\varepsilon}.
 \end{eqnarray*}   
The bound of and $\hbox{II}_3$ is done similarly.
\end{proof}

\end{document}